\documentclass[11pt,leqno]{article}

\usepackage{amsthm,amsfonts,amssymb,amsmath,color}

\numberwithin{equation}{section}

\renewcommand\d{\partial}
\renewcommand\a{\alpha}
\renewcommand\b{\beta}
\renewcommand\o{\omega}
\newcommand\s{\sigma}
\renewcommand\t{\tau}
\newcommand\R{\mathbb R}\newcommand\N{\mathbb N}\newcommand\Z{\mathbb Z}
\newcommand\C{\mathbb C}

\def\g{\gamma}

\def\t{\tau}

\def\th{\theta}

\def\l{\lambda}

\def\epsilon{\varepsilon}
\def\e{\varepsilon}

\def\lxi{\langle \xi \rangle}

\newcommand\br{\begin{rem}}
\newcommand\er{\end{rem}}
\newcommand\bp{\begin{pmatrix}}
\newcommand\ep{\end{pmatrix}}
\newcommand\be{\begin{equation}}
\newcommand\ee{\end{equation}}
\newcommand\ba{\begin{equation}\begin{aligned}}
\newcommand\ea{\end{aligned}\end{equation}}

\newcommand\nn{\nonumber}

\newcommand{\Id}{{\rm Id }}

\newcommand{\tr}{{\rm tr }}

\newtheorem{defi}{Definition}[section]
\newtheorem{theo}[defi]{Theorem}
\newtheorem{prop}[defi]{Proposition}
\newtheorem{lem}[defi]{Lemma}

\newtheorem{rem}[defi]{Remark}
\newtheorem{ass}[defi]{Assumption}

\numberwithin{equation}{section}

\begin{document}

\title{Partially strong transparency conditions and a singular localization method in geometric optics}

\author{Yong Lu\footnote{Mathematical Institute, Faculty of Mathematics and Physics, Charles University, Sokolovsk\'a 83, 186 75 Praha, Czech Republic, {\tt luyong@karlin.mff.cuni.cz}}
\,\,and Zhifei Zhang\footnote{School of Mathematical Science, Peking University, Beijing 100871, P. R. CHINA, {\tt zfzhang@math.pku.edu.cn}}  }

\date{}

\maketitle

\begin{abstract}
This paper focuses on the stability analysis of WKB approximate solutions in geometric optics with the absence of strong transparency conditions under the terminology of Joly, M\'etivier and Rauch.  We introduce a compatible condition and a singular localization method which allows us to prove the stability of WKB solutions over long time intervals. This compatible condition is weaker than the strong transparency condition. The singular localization method allows us to do delicate analysis near resonances. As an application, we show the long time approximation of Klein-Gordon equations by Schr\"odinger equations in the non-relativistic limit regime.

\end{abstract}

{\bf Keywords}:

Transparency condition, stability, WKB solution, singular localization method.

 \tableofcontents

\section{Introduction}\label{sec:intr}

In this paper, we consider the long time behavior of the solutions to Cauchy problems for symmetric hyperbolic systems of the following form
\be\label{00}
\left\{\begin{aligned}
&\d_t U+\frac{1}{\e}A(\d_x) U +\frac{1}{\e^2} A_0 U=B(U,U),\\
&U(0,\cdot)\in H^s(\R^d),
\end{aligned}\right.
\ee
where $U(t,x):\R_+\times \R^d \to \R^N$ is the unknown, $A(\d_x)=\sum_{j=1}^d A_j \d_{x_j}$ with $A_j,\ j=1,\cdots,d$ real-valued symmetric matrices, $A_0$ is a real-valued skew-symmetric matrix and $B(\cdot, \cdot): \R^N\times \R^N \to \R^N$ is a symmetric bilinear application.  The matrices $A_j$ are all of order $N\times N$. The initial datum $U(0)$ is supposed to be in Sobolev space $H^s$ with $s$ sufficiently large.

 \subsection{Setting and background}

 We will consider solutions of \eqref{00} having the from
\be\label{sol-form}
U(t,x)=e^{-i\o t/\e^2}U_{0,1}(t,x) + e^{i\o t/\e^2} \overline U_{0,1}(t,x)+O(\e),
\ee
which is highly oscillating in time with $\o$  an appropriate characteristic temporal frequency satisfying
\be\label{dis-rel}
{\rm det}\,(-i\o+A_0) =0,
\ee
which is the so called  {\em dispersion relation}; $\o$ is also called the temporal wave number.

The study of highly oscillating solutions to hyperbolic systems falls in the framework of {\em geometric optics}. Considerable progress has recently been made in this field, especially following the works of Joly, M\'etivier and Rauch in the nineties (see for instance \cite{JMR0,DJMR, JMR1, JMR2,CD-Duke}, and \cite{Dumas1} for an overview and references therein). In geometric optics, the main issue is the stability of a family of approximate solutions, namely \emph{WKB solutions}, and the main obstacle is the resonance. 

The hyperbolic system in \eqref{00} is symmetric semilinear. Then with $H^s,\ s>d/2$ initial data, the local well-posedness is classical (see \cite{Ma} or \cite{Me}). In spite of the presence of the large prefactors $1/\e$ and $1/\e^2$, the uniform $H^s$ estimate still holds due to the symmetry of $A_j$ and the skew-symmetry of $A_0$. Hence, with initial data that are uniformly bounded in $H^s, \ s>d/2$, the classical existence time to Cauchy problem \eqref{00} is $O(1)$.

Our goal in this paper is to study the behavior of the solution to \eqref{00} beyond the classical time $O(1)$ up to long time of orders $O(1/\e^\g)$ for some $\g>0$ given $O(1)$ initial data. This study falls in the framework beyond the weakly nonlinear regime of geometric optics, thus the classical results, for instance \cite{JMR0} -- geometric optics for $O(1)$ amplitude, but $O(1)$ time, as well as \cite{JMR1} -- diffractive optics for $O(1/\e)$ time, but $O(\e)$ amplitude, do not apply.  By assuming the  global-in-time (or long time)  existence of approximate solutions, we exhibit some sufficient conditions on \eqref{00},  and introduce a singular localization method which allows us to make use of such sufficient conditions to show the existence as well as the stability of solutions over long time intervals. Such sufficient conditions are described in Section \ref{sec:ass-res}, in particular in the key Assumption \ref{ass-ps-tran}. The singular localization method is introduced and described in Section \ref{sec:pf-1st} and Section \ref{sec:pf-2nd}.

As an application,  we show in Section \ref{sec:motivation} that in the non-relativistic limit regime the quadratic Klein-Gordon equation can be well approximated by linear Schr\"odinger equations over long time intervals of order $O(1/\e)$. 

\medskip

We point out that the condition imposed in the key Assumption \ref{ass-ps-tran} is analogous to, but weaker and more general than, the {\em strong transparency condition} exhibited by Joly, M\'etivier and Rauch in \cite{JMR2}. The strong transparency condition allows a control of the constructive interaction of characteristic waves at the resonances by a normal form reduction, thus leading to the stability of approximate WKB solutions. The transparency condition is analogous to the null conditions introduced by Klainerman in \cite{kl}; the normal form reduction allowed by the transparency property is analogous to the analysis of Shatah in \cite{shatah}.  As it will be shown in Section \ref{sec:very-KG},  the quadratic Klein-Gordon equation satisfies the condition imposed in Assumption \ref{ass-ps-tran} while it does not satisfy the strong transparency condition.

\label{page-em2}We also point out that the {\em approximate linear transparency condition} introduced in \cite[Assumption 1.7]{em2}, which is also weaker than the strong transparency condition, has similarities with our setting. To be precise,  the condition in \cite[Assumption 1.7]{em2} can be recovered by taking $\a=1/2$ in our Assumption \ref{ass-ps-tran}. We remark that in our setting, $\e$ corresponds to $\sqrt\e$ in \cite{em2}. Moreover, the idea to decompose the integral form into two parts (see page 31 in \cite{em2}), where one part is the integral over a neighbourhood of resonances $D^\e:=\{\eta'\in \R^d: \ |\psi^\e(\e \eta')\leq \sqrt\e|\}$ and the other part is the integral over the complement of $D^\e$,  is essentially of the same sprit as our singular localization method.  However, the analysis here is not simply a generalization of the argument in \cite{em2} from $\a=1/2$ to general $\a>0$. In particular, the singular localization method used in this paper can be employed to deal with the Klein-Gordon-wave equations (1.10) and (1.11) studied in \cite{em2} and to show the same stability results.  But the analysis in \cite{em2}  strongly relies on the typical structure of the system which we do not assume in this paper (see equation (1.9)). In particular, the block diagonal structure of the differential operator and the special coupling structure of the nonlinear terms play a crucial role for the stability analysis argument in \cite{em2}.

\medskip

The strong transparency condition ensures the stability of WKB solutions.\label{page-case-nt} However, many (most) physical models in geometric optics do not fulfill the strong transparency condition, such as the Euler-Maxwell system, the Klein-Gordon system, the Maxwell-Landau-Lifshitz system, the Klein-Gordon-Zakhorov system, etc.. Thus, the study of the case where the strong transparency condition is not satisfied is highly important. In \cite{em4}, Texier and the first author gave a systematic study  concerning the case with the absence of the strong transparency condition.  In particular, the article \cite{em4} contains a detailed account of how resonances may destabilize the WKB solutions. There was exhibited an almost sufficient and necessary condition for the stability of WKB solutions by giving a scalar index $\Gamma$ of which the positivity ensures instability and the negativity ensures stability.

However, the case $\Gamma=0$ is not included in the study of \cite{em4}. The case considered in this paper corresponds to a large family of subcases of the case $\Gamma=0$. Even if the scaling in this paper is different from the one in \cite{em4}, the result obtained, as well as the method used in this paper could give some clear clues for the study in the scaling of \cite{em4} and others.

\subsection{Assumptions and  main results}\label{sec:ass-res}

In this section, we state our main assumptions and results.

 \subsubsection{Smooth spectral decomposition}

 We first assume the symbol of the differential operator on the left-hand side of \eqref{00} admits a smooth spectral decomposition:
\begin{ass}\label{ass-spec}
We assume that the spectral decomposition
$$
A(\xi)+A_0/i=\sum_{j=1}^J \l_j(\xi)\Pi_j(\xi)
$$
is smooth, meaning that the eigenvalues $\l_j(\xi)$ and  the eigenprojectors $\Pi_j(\xi)$ are smooth in $\xi\in \R^d$. Moreover, for any $1\leq j\leq J$, we suppose that $\l_j(\cdot)$ and $\Pi_j(\cdot)$ are in the classical symbol class $S^1$ and $S^0$, respectively.

\end{ass}

The definition of the symbol classes $S^m$ is classical and will be recalled  in Section \ref{sec:semi-Fourier}.

\subsubsection{WKB solutions}\label{sec:ass-app}

By WKB (approximate) solutions of \eqref{00} we mean truncated power series in $\e$, where each term in the series is a trigonometric polynomial in $\th:=-\o t/\e^2,$ that approximately solves \eqref{00}.
Precisely,  a WKB solution $U_a$ has the form
 \be\label{suite}U_a(t,x) = \sum_{n=0}^{K_a+1} \e^{n} {\bf U}_{n}(t,x,\th), \quad {\bf U}_{n} (t,x,\th)= \sum_{p \in {\cal H}_n} e^{i p \th} U_{n,p}(t,x), \  K_a \in \Z_+, \,\, {\cal H}_n \subset \Z,\ee
which solves
\be \label{wkb-new1}
 \left\{ \begin{aligned} &\d_t U_a + \frac{1}{\e} A(\d_x) U_a +\frac{1}{\e^2} A_0 U_a   = B(U_a ,U_a ) - \e^{K_a} R^\e,\\
 & U_a (0,x)=U(0,x) - \e^K \psi^\e(x)\end{aligned}\right.
 \ee
 with $(R^\e,\psi^\e)$  bounded uniformly in $\e$ in some Sobolev spaces. Parameters $K_a$ and $K$ describe the level of precision of the WKB solution $U_a .$ Here ${\cal H}_n$ are the harmonics sets. In particular, in this paper, the leading harmonics set is defined as
 \be\label{def-H0}
 {\cal H}_0:=\{ -1,1 \}\subset  {\mathcal R}:=\big\{p: {\rm det}\,(-ip\o+A_0)=0 \big\}.\nn
 \ee

The idea to find or construct such a WKB solution is quite straightforward, that is to plug a solution $U_a$ of the form \eqref{suite} into the system \eqref{00}, and then consider the equations at each order $\e^n, \ n=-2,-1,\cdots.$ If one can solve the equations of order $\e^n$ up to some positive order $N_a$, then one can solve the original system \eqref{00} approximately, up to a small remainder of order $O(\e^{N_a+1})$.

\medskip

In this paper, we assume that there exits a global-in-time approximate solution for \eqref{00}.
\begin{ass}\label{ass-app} Let $s>d/2$. We assume the vector space $\ker(-i\o +A_0)$ is of dimension one with $e_1$ a generator of norm one.  We assume there exists $U_a\in C_b \left([0,\infty); H^{s+1}\right)\cap C_b^1 \left([0,\infty); H^{s}\right)$ solving \eqref{wkb-new1}  for all $(t,x)\in (0,\infty)\times \R^d$ with $K_a=2,~K=1$, and there holds the estimate
\be\label{bound-R-Psi}
\sup_{0<\e<1} \left(\|R^{\e}\|_{L^\infty\left(0,\infty; H^{s}\right)}+\|\psi^{\e}\|_{ H^{s}}\right)<+\infty.
\ee
Moreover, $U_a$ is of the form \eqref{suite} with $U_{n}\in C_b \left([0,\infty); H^{s+1}\right)\cap C_b^1 \left([0,\infty); H^{s}\right),~0\leq n\leq K_a+1=3$; in particular, the leading term $U_0$ is of the form
\be\label{ass-u0}
 U_0=e^{-i\o t/\e^2} U_{0,1}+e^{i \o t/\e^2} U_{0,-1},
 \ee
where
\be\label{ass-U01} U_{0,1}(t,x)=g_1(t,x) e_{1}, \quad U_{0,-1}(t,x)=g_{-1} (t,x) e_{-1}, \quad g_{-1}:=\bar g_1, \ e_{-1}=\bar e_1
\ee
for some scalar function $g_1\in C_b \left([0,\infty); H^{s+1}\right)\cap C_b^1 \left([0,\infty); H^{s}\right)$.
\end{ass}

The notation $\bar a$ stands for the complex conjugate of $a$.

\begin{rem}\label{rem-assapp1}
To obtain our main result Theorem \ref{thm:general}, the existence time and uniform bound for $U_a$ in Assumption \ref{ass-app} can be generalized to $U_a\in C_b \left([0,\frac{T}{\e}]; H^{s+1}\right)\cap C_b^1 \left([0,\frac{T}{\e}]; H^{s}\right)$ satisfying the uniform estimate
\be\label{uniest-ua}
\|U_n\|_{L^\infty\left([0,\frac{T}{\e}]; H^{s+1}\right)}+\|\d_t U_n\|_{L^\infty\left([0,\frac{T}{\e}]; H^{s}\right)}\leq C<\infty \nn
\ee
for some constant $C$ independent of $\e$  and some time $T>0$ independent of $\e$.
\end{rem}

The local-in-time WKB solutions to \eqref{00} can be constructed by using standard WKB expansion under the constrain \eqref{ass2-cond20} given later on. The main point of Assumption \ref{ass-app} is the global-in-time (or long time) existence and global-in-time (or long time) uniform bounds for the approximate solutions.

In the sequel of this section, we impose some compatibility conditions which ensure the existence of global-in-time approximate solutions such that Assumption  \ref{ass-app} is satisfied.

\medskip

{\bf Condition 1:} The leading terms of the initial data satisfy:
$$
U(0)=U_{0,1}(0,x)+\overline U_{0,1}(0,x)+O(\e) \ \mbox{in} \ H^s,\quad U_{0,1}(0,x)\in \ker (-i\o+A_0).
$$
This is often called the {\em polarization condition.}

\medskip

Let $\pi_p$ be the orthogonal projection onto $\ker (-ip\o+A_0)$ and $L_p^{-1}$ be the (partial) inverse of $L_p:=(-ip\o+A_0)$ such that
\be\label{ass2-cond1}
\pi_p L_p^{-1}=L_p^{-1}\pi_p=0,\quad  L_p L_p^{-1}=L_p^{-1}L_p=\Id-\pi_p. \nn
\ee

{\bf Condition 2:}
We suppose for any $p\in \Z$ and any $\xi\in \R^d$ there holds
\be\label{ass2-cond20}
\pi_p A(\xi) \pi_p =0.
\ee

\medskip

{\bf Condition 3:}
We suppose for any $p\in \Z$  there holds
\be\label{ass2-cond2}
 \pi_p \sum_{p_1+p_2=p}B(\pi_{p_1},\pi_{p_2})=0.
\ee

\medskip

{\bf Condition 4:} We suppose furthermore for any $p\in \Z$ and any $\xi\in \R^d$ that
\ba\label{ass2-cond3}
&\pi_p A(\xi)L_p^{-1} A(\xi)L_p^{-1} A(\xi)\pi_p =0,\\
&\pi_p A(\xi) L_{p}^{-1} \sum_{p_1+p_2=p}B(\pi_{p_1},\pi_{p_2})+2 \pi_p  \sum_{p_1+p_2=p}B(\pi_{p_1}, L_{p_2}^{-1}A(\xi)\pi_{p_2})= 0.\nn
\ea

We then have

\begin{prop}\label{prop:Ua}

Assumption \ref{ass-app} holds true if {\bf Condition 1}, {\bf Condition 2}, {\bf Condition 3} and {\bf Condition 4} are all satisfied. More precisely, we have

\begin{itemize}
\item[(i).] Under {\bf Condition 1} and {\bf Condition 2} and the additional assumption:
\ba\label{con-wkb1}
\ker (-ip\o+A_0) = \{0\}\ \mbox{for any $p$ satisfying $|p|\geq 2$};\quad
\pi_0 B(\pi_{1},\pi_{-1}) =0,
\ea
one can construct a uniformly bounded local-in-time WKB solution $U_a$ solving \eqref{wkb-new1} with arbitrary $K_a$ and $ K$.

 \item[(ii).]   Under {\bf Condition 1}, {\bf Condition 2} and {\bf Condition 3}, one can construct a uniformly bounded global-in-time WKB solution $U_a$ solving \eqref{wkb-new1} with $K_a=K=1$.

 \item[(iii).]   Under {\bf Condition 1}, {\bf Condition 2}, {\bf Condition 3} and {\bf Condition 4},  one can construct a uniformly bounded global-in-time WKB solution $U_a$ solving \eqref{wkb-new1} with $K_a=2,\ K=1$.

 \end{itemize}

\end{prop}

The proof of Proposition \ref{prop:Ua} can be done by employing the standard WKB expansion for which we give a detailed description in Section \ref{sec:wkb-sol}. The WKB expansion in Section \ref{sec:wkb-sol} is done for a specific example instead of the general case, but the procedure is essentially the same. So here we omit the proof of Proposition \ref{prop:Ua}. We point out that in statements (ii) and (iii) in Proposition \ref{prop:Ua}, we do not need to assume the additional assumption \eqref{con-wkb1} to make sure the leading term of the approximate solution has the form \eqref{ass-u0}. Indeed, {\bf Condition 3}, together with {\bf Condition 1}, allows us to choose trivial solutions $U_{0,p}\equiv 0$ for any $p\not\in \{-1,1\}$ in the WKB expansion.
\medskip

Finally we give a remark concerning the conditions exhibited above.
\begin{rem}\label{rem-assapp3}
Concerning condition \eqref{ass2-cond20}, it was shown in \cite{Lax1, DJMR} (see also \cite[Proposition 2.6]{T-st} a unified proof for such algebraic lemmas) that for any $\xi\in \R^d$, there holds
\be\label{ast-alg}
\pi_p A(\xi) \pi_p = \nabla_\xi \l_{j_{p}}(0)\cdot \xi,
\ee
where $\l_{j_{p}}$ is the eigenmode in Assumption \ref{ass-spec} such that $\l_{j_{p}}(0)=-p\o$. Thus, condition \eqref{ass2-cond20} means $\nabla_\xi \l_{j_{p}}(0)=0$. This associates with the condition in Assumption 1.5 in \cite{em2} saying that $(-p\o,0)$ is a local extremum of every branch of the characteristic defined in \eqref{char0}.

The condition in \eqref{ass2-cond2} corresponds exactly to the weak transparency condition introduced by Joly, M\'etivier and Rauch in \cite{JMR2}. See also \eqref{wtrans0} for the precise description.

\end{rem}

\subsubsection{Partially strong transparency}
Now we give our key assumption:
\begin{ass}\label{ass-ps-tran} For any $p\in \{-1,1\}$ and any $1\leq j,j'\leq J $, there exists some constant $C$ and $0< \a_{j,j',p}\leq 1$ such that
\begin{equation}\label{sttrans-new}
\big|\Pi_j(\xi)B(e_p)\Pi_{j'} (\xi) \big|\leq C|\l_j(\xi)-\l_{j'}(\xi)-p\o|^{ \a_{j,j',p}},\quad \mbox{for all $\xi\in \R^d$}.
\end{equation}
The vectors $e_{1}$ and $e_{-1}$ are introduced in Assumption \ref{ass-app}. The linear operator $B(e_p)$ is defined as $B(e_p)V:=B(e_p,V)$ for any $V\in \C^N$.

Moreover, for any $p\in \{-1,1\}$ and any $1\leq j,j'\leq J $, the resonance set
\be\label{def-re-set0}
R_{j,j',p} : =\{\xi\in\R^d : \l_j(\xi) - \l_{j'}(\xi) -p \o =0\}
\ee
is compact. If $R_{j,j',p} = \emptyset$, there exits $c_{j,j',p}>0$ such that
\be\label{def-cjjp}
|\l_j(\xi) - \l_{j'}(\xi) -p \o | \geq c_{j,j,p} \quad \mbox{for all $\xi \in \R^d$}.
\ee
\end{ass}

Given Assumption \ref{ass-ps-tran}, we further define the exponent:
\be\label{def-al}
\a:=\min_{j,j',p}\a_{j,j',p}.
\ee

If $\a=1$, Assumption \ref{ass-ps-tran} becomes the strong transparency assumption (see \cite{JMR2} and Section \ref{sec:con-tra} later on). If $\a=1/2$, Assumption \ref{ass-ps-tran} implies \cite[Assumption 1.7]{em2}. Because of the presence of the fractional power $0<\a\leq 1$,  we may call such condition imposed in Assumption \ref{ass-ps-tran} as the {\em partially strong transparency condition}.

We will show that under Assumption \ref{ass-ps-tran}, the approximate solution assumed in Assumption \ref{ass-app} is stable up to time of order $O(1/\e^{\a})$. Now we give an additional assumption, which allows us to show the stability up to even longer time $t_\e$ which is of order
\be\label{te-order}
t_\e=O(1/\e^{2\a}), \ \mbox{if $\a\leq 1/2$}; \quad t_\e=O(1/\e), \ \mbox{if $\a\geq 1/2$}. \nn
\ee
\begin{ass}\label{ass-ps-tran-add} If for some $(j,j',p)$, the component $\a_{j,j',p}$ in \eqref{sttrans-new} cannot be chosen equal to 1, we assume that either $\l_{j}$ or $\l_{j'}$ is identically a constant.
\end{ass}

The case $\a_{j,j',p}<1$ corresponds to the case where the interaction coefficient $\Pi_j(\xi)B(e_p)\Pi_{j'} (\xi)$ is not strongly transparent. The nontransparent interaction coefficients (or the resonances) happen quite often between two eigenmodes involving a zero eigenmode. Thus,  Assumption \ref{ass-ps-tran-add} is natural in such a sense.

\subsubsection{Main result}
We are ready to state our main theorem:
\begin{theo}\label{thm:general} Let $s>d/2$ and $0<\e<\e_0$ with $\e_0$ sufficient small. Under Assumption \ref{ass-spec}, Assumption \ref{ass-app} and Assumption \ref{ass-ps-tran}, the Cauchy problem \eqref{00} admits a unique solution $U\in L^\infty([0,\frac{T_1}{\e^\a}];H^s)$ for some $T_1>0$ independent of $\e$. Moreover, there holds the error estimate
\be\label{est-error1}
\|U-U_a\|_{L^\infty\left([0,\frac{T_1}{\e^\a}]; H^{s}\right)}\leq C\, \e.
\ee

\smallskip

If in addition Assumption \ref{ass-ps-tran-add} is satisfied, the solution $U\in L^\infty([0,\frac{T_2}{\e^{ \a_1}}];H^s)$ where $T_2 > 0$ is independent of $\e$ and
\be\label{def-a1}
\a_1 := \min\{2\a,1\}.\nn
\ee
Moreover, there holds
\be\label{est-error2}
\|U-U_a\|_{L^\infty\left([0,\frac{T_2}{\e^{\a_1}}]; H^{s}\right)}\leq C\, \e.
\ee

Here $C$ is a constant independent of $\e$ and the number $\a$ is defined in \eqref{def-al}.

\end{theo}

We remark that Theorem \ref{thm:general} shows a {\em linear stability} phenomenon.
\begin{rem}
Theorem \ref{thm:general} gives an existence and stability result beyond the classical existence time. In Assumption \ref{ass-app}, the initial difference between the exact solution and the approximate solution is of order $O(\e)$. The estimates \eqref{est-error1} and \eqref{est-error2} imply that the error stays of order $O(\e)$ over long time intervals considered. This means that the approximate solution is {\em linearly stable} over the corresponding long time intervals, in the sense that the initial error is not much amplified through the dynamics of the system.
\end{rem}

\subsection{Structure of the paper}\label{sec:tructure}

In Section \ref{context}, we introduce some context of our study in geometric optics and we also emphasize the novelty of the study in this paper. In Section \ref{sec:semi-Fourier}, we recall the concept of semiclassical Fourier multipliers and the action estimates including a commutator estimate.  Section \ref{sec:stability}, Section \ref{sec:pf-1st} and Section \ref{sec:pf-2nd} are devoted to the proof of Theorem \ref{thm:general}. In Section \ref{sec:motivation}, we give an application of our study in the non-relativistic limit problem of Klein-Gordon equations.

In the sequel, we use $C$ to denote a positive constant independent of $\e$. However, the value of $C$ may change from line to line.

\section{Transparency conditions and stabilities}\label{context}

 In this section, we first recall some basic concepts in geometric optics including transparency conditions exhibited in \cite{JMR2} by Joly, M\'etivier and Rauch,  and the normal form method used to obtain the stability of WKB approximate solutions. We then briefly recall the study in \cite{em4} and explain why the study of this paper is important for stability analysis in geometric optics, particularly in  completing the program of \cite{em4} and \cite{JMR2}.


\subsection{Weak transparency}\label{sec:con-wkb}

In Section \ref{sec:ass-app}, we gave the definition of WKB solutions and we assumed the existence of WKB solutions in Assumption \ref{ass-app}.  In \cite{JMR1,JMR2}, Joly, M\'etivier and Rauch exhibited the {\em weak transparency condition} that allows one to construct WKB approximate solutions. Before stating such weak transparency condition, we introduce some basic concepts.

We define the characteristic variety of the differential operator in \eqref{00}:
 \be\label{char0}
 {\rm Char}:=\{ (\t,\xi): {\rm det}\,\big(-i\t +A(i\xi)+A_0\big)=0\}.
 \ee
Given a couple $(\t,\xi)$, we denote by $\Pi(\t,\xi)$  the orthogonal projector onto
\be\label{char1}\ker \big(-i\t +A(i\xi)+A_0\big).\nn
\ee

 We fix a basic characteristic space-time vector
\be\label{def-beta}
\b:=(\o,k)\in \R\times \R^d\nn
\ee
satisfying the dispersion relation
$$
\det (-ip\o+A(ik)+A_0)=0,
$$
where $k$ is called the {\em spatial wave number} and $\o$ is called the {\em temporal wave number}.

 Remark that in this paper, the spatial wave number $k$ is assumed to be zero, since we are not considering solutions that are highly oscillating in spatial variable (see  \eqref{sol-form}). However, to introduce the general concepts concerning transparency conditions, we take general $\b=(\o,k)$.


Now we can state the \emph{weak transparency} condition introduce in \cite{JMR2}:\medskip

{\bf Weak transparency.} For any $p,p_1\in \Z$ and any $U,~V\in \C^N$, one has
\begin{equation}\label{wtrans0}
\big|\Pi(p_1 \b)B\big(\Pi((p_1-p) \b)U,\Pi(p \b)V\big)\big|=0.
\end{equation}

We find that this weak transparency condition corresponds to exactly  the condition introduced in (\ref{ass2-cond2}).

\subsection{Strong transparency and normal form method}\label{sec:con-tra}

Given a WKB solution, a nature question is the stability property of this WKB solution. To this issue, one turns to consider the perturbed system.  Let $U$ and $U_a$ be the exact solution and the WKB solution which solve \eqref{00}  and \eqref{wkb-new1} respectively. Then the perturbation
\be\label{def-dot-U0}
\dot U:=\frac{U-U_a}{\e}\nn
\ee
solves
 \be \label{eq-dotU0}
 \left\{ \begin{aligned} &\d_t \dot U + \frac{1}{\e} A(\d_x) \dot U +\frac{1}{\e^2} A_0 \dot U = 2B(U_a)\dot U+\e B(\dot U,\dot U) + \e^{K_a-1} R^\e,\\
 & \dot U(0,x)=\e^{K-1} \psi^\e(x).\end{aligned}\right.
 \ee

An advantage of considering the perturbed system \eqref{eq-dotU0} is that the nonlinear term is small of order $\e$. The leading term becomes the linear one $2B(U_a )\dot U$. However, even when the parameter $K_a$ and $K$ are sufficiently large and the WKB solution $U_a$ is uniformly bounded in proper Sobolev spaces and solves \eqref{wkb-new1} globally in time, the classical existence time $T_\e^*$ to \eqref{eq-dotU0} is at most of logarithmic order:
$$
\dot T_\e^* \geq T_0 |\ln \e|,\quad \mbox{ for some $T_0>0$ independent of $\e$}.
$$
This logarithmic order existence time can be achieved by employing the argument in \cite{Lannes-Rauch} as well as in \cite{CD_DS}.

To achieve an even larger scale of the maximal existence time such as
\be\label{dot-T-*}
\dot T_\e^* \geq \frac{T}{\e^\g},\quad \mbox{for some $T>0,\ \g>0$ independent of $\e$},\nn
\ee
as well as the uniform boundedness of the perturbation over such long time, one needs to make use of more structure of the system  \eqref{eq-dotU0}. To this end, also in \cite{JMR2}, Joly, M\'etivier and Rauch introduced the strong transparency condition that allows them to  eliminate the linear leading term $2B(U_a )\dot U$ up to a remainder of order $\e$ by using a normal form method. If this can be done,  the right-hand side of $\eqref{eq-dotU0}_1$ becomes of order $\e$ and the well-posedness over time of order $1/\e$ follows from the classical theory. We recall the strong transparency condition:\medskip

\label{JMR-srong-trans}{\bf Strong transparency.} There exists a constant $C$ such that for any $p\in \Z$,  $1\leq j,j'\leq J $, $\xi\in \R^d$ and  $U,~V\in \C^N$, one has
\begin{equation}\label{sttrans0}
\big|\Pi_j(\xi+p k)B\big(\Pi(p \b)U,\Pi_{j'} (\xi)V\big)\big|\leq C|\l_j(\xi+p k)-\l_{j'}(\xi)-p\o|\cdot|U|\cdot|V|.
\end{equation}

In the above inequality \eqref{sttrans0}, the terms $\Pi_j(\xi+p k)B\big(\Pi(p \b),\Pi_{j'} (\xi)\big)$ on the left-hand side are named {\em interaction coefficients}, and the factors $\l_j(\xi+p k)-\l_{j'}(\xi)-p\o$ on the right-hand side are called {\em interaction phases}. The frequencies $\xi$ such that $\l_j(\xi+p k)-\l_{j'}(\xi)-p\o=0$ are named {\em resonances} and the $(j,j',p)$-resonance set is defined as
\be\label{def-Rjjp}
R_{j,j',p}:=\{\xi\in\R^d,\,\l_j(\xi+p k)=p \o+\l_{j'}(\xi)\}.
\ee

The equalities $\l_j(\xi+p k)-\l_{j'}(\xi)-p\o=0$ are named {\rm resonance equations}.

The strong transparency condition offers a control of the quantity
$$
\frac{\Pi_j(\xi+p k)B\big(\Pi(p \b)U_{0,p},\Pi_{j'} (\xi)\big)}{\l_j(\xi+p k)-\l_{j'}(\xi)-p\o}
$$
which appears in the normal form reduction. The interaction phase plays the role of divisor.

\medskip

The method of a normal form reduction is essentially a change of unknown which can be linear or nonlinear. In general, the nonlinear normal form method needs more constrains on the structure of the equations than the linear one. In our setting, we are trying to eliminate the linear leading term, it is possible to use the linear version of the normal form method. The idea is to consider a change of unknown of the following form
\be\label{change1-int}
\dot U_1=\left(\Id +\e^2 M\right)^{-1} \dot U,\nn
\ee
with $M$  to be determined. Then the system in $\dot U_1$ is of the form
\ba\label{eq-dot-U1-int}
&\d_t \dot U_1  + \frac{i}{\e^2} \mathcal{A}(\e D_x) \dot  U_1 =\left(2B(U_a )-i[{\mathcal A}(\e D_x),M] \right)+\e \mathcal{R}_1,\nn
\ea
where $\e \mathcal{R}_1$ contains all the terms formally of order $O(\e)$ and
$$
D_x:={\d_x}/{i},\quad {\mathcal A}(\xi):= A(\xi) + A_0/i.
$$
The goal is to find a proper operator $M$ such that the $O(1)$ term on the right-hand side of \eqref{eq-dot-U1-int} is eliminated with a small remainder. It is shown in \cite{JMR2}, as well as in \cite{em3,MLL, em4, slow}, that
such $M$ can be well defined provided the strong transparency condition is satisfied, and the linear leading term can be eliminated with an $O(\e)$ remainder.

\medskip

Such strong transparency condition is satisfied for some physical models, such as the Maxwell-Bloch system (see \cite{JMR2}) and the one-dimensional Maxwell-Landau-Lifshitz system (see \cite{MLL}). Moreover, Texier showed that the Euler-Maxwell
equations satisfy a form of transparency \cite{em3}, Cheverry, Gu\`es and M\'etivier \cite{CGM} showed  that for systems of conservation laws, linear degeneracy of a field implies transparency. However, many (most) physical models in geometric optics do not fulfill the strong transparency condition, such as the Klein-Gordon system considered in Section \ref{sec:motivation} later on.

\smallskip

 In this paper, we impose a weaker condition in Assumption \ref{ass-ps-tran} compared to the strong transparency condition. A key novelty of our study is to extend the long time stability analysis in geometric optics under  Assumption \ref{ass-ps-tran} without assuming the strong transparency condition. Another novelty is to introduce a singular localization method,  which allows us to do delicate analysis for the interaction coefficients near resonances in order to obtain long time stability.

\subsection{Absence of strong transparency}\label{sec:con-not}

 In \cite{em4}, Texier and the first author give a systematic study for the case where the strong transparency condition fails to be satisfied for semilinear hyperbolic systems of the following form:
\be \label{eq-slow} \d_t U + \frac{1}{\e} A_0 U + \sum_{1 \leq j \leq d} A_j \d_{x_j} U = \frac{1}{\sqrt \e} B(U,U),\ee
where the constant matrix $A_0$ is non-zero and skew-symmetric and the matrices $A_j, \ 1\leq j\leq d$ are constant and symmetric. Highly oscillating initial data are considered:
\be \label{data-slow}
 U(0,x) = \Re e \, \big( a(x) e^{i k \cdot x/\e} \big) + \sqrt\e \varphi^\e(x).\nn
 \ee
Here $k$ is the spatial wave number. Let $\o$ be a temporal wave number satisfying the dispersion relation:
$$
{\rm det}\, (-i\o +A(ik)+A_0)=0.
$$

The absence of strong transparency means that there
exists $(j,j',p)$ such that \eqref{sttrans0} is not satisfied.
Denote $J_0$  the set containing all such indices $(j,j',p)$ and $R_{j,j',p}$ the $(j,j',p)$-resonant set defined as in \eqref{def-Rjjp}. If $R_{j,j',p}$ is empty, by the regularity of $\l_j$ and $\Pi_j$, $j=1,\cdots, J$, the strong transparency condition \eqref{sttrans0} is satisfied for the index
$(j,j',p)$. Then for any $(j,j',p)\in J_0$, $R_{j,j',p}$ is not
empty, and the following quantity is well defined:
\be\label{def:Gamma} \Gamma:=\sup_{(j,j',p)\in J_0}|g_p(0,x_{p})|^2\sup_{\xi\in
R_{j,j',p}} \tr\Big(\Pi_j(\xi+p k)B(\vec
e_p)\Pi_{j'}(\xi)B(\vec e_{-p})\Pi_j(\xi+p k)\Big), \nn\ee
where $g_p$ comes from the polarization condition
$$
U_{0,p}(t,x)= g_p(t,x)e_p,\quad e_p\in \ker (-ip\o +A(ik)+A_0),
$$
and $x_p$ is a point where $|g_p(0,\cdot)|$ admits its maximum. Here $U_{0,p}$ are the leading terms of the WKB solution. In \cite{em4}, it is shown that the stability  of  the WKB solution
is determined by the sign of $\Gamma$:\\
If  $\Gamma<0$, the perturbation  system is symmetrizable and the WKB solution is stable.\\
If $\Gamma>0$, it is shown that the WKB solution is unstable.

However, the degenerate case $\Gamma =0$ is not included in the study of \cite{em4}. In \cite{slow}, the first author considered a subcase of $\Gamma=0$, that is the case $g_p(0,x)=0$ for any $(j,j',p)\in J_0$. Under the assumptions $\d_t g_p(0,x)\neq 0$ and the positivity of the following quantity
\be\label{nt-slow}
\tilde\Gamma:=\sup_{(j,j',p)\in J_0}\sup_{\xi\in
R_{j,j',p}} \tr\Big(\Pi_j(\xi+p k)B(\vec
e_p)\Pi_{j'}(\xi)B(\vec e_{-p})\Pi_j(\xi+p k)\Big),\nn
\ee
the instability are discovered instantaneously, even though the equations linearized around the leading WKB terms are initially stable.

The study of this paper corresponds to a large subcase of $\Gamma=0$ which goes through the case $\tilde \Gamma =0$ under our key Assumption \ref{ass-ps-tran}.  Indeed, Assumption \ref{ass-ps-tran} states that, near resonances, the interaction coefficients $\Pi_j(\xi+p k)B\big(\Pi(p \b),\Pi_{j'} (\xi)\big)$ cannot be controlled by the resonant phase $|\l_j(\xi+p k)-\l_{j'}(\xi)-p\o|$, but rather are controlled by some fraction power of the resonant phase $|\l_j(\xi+p k)-\l_{j'}(\xi)-p\o|^\a, 0<\a\leq 1$. Even the scaling of this paper is different from that in \cite{em4}, the idea introduced in this paper may be well employed.

\section{Semiclassical Fourier multipliers}\label{sec:semi-Fourier}

In this section, we introduce the basic concepts about semiclassical Fourier multipliers, in particular the commutator estimates between a semiclassical Fourier multiplier and a scalar function multiplier. This will be needed throughout the paper. In the sequel of this paper, the function $g$, or $g_p$ in the next sections, is also considered as the operator which consists in the multiplication by this function.

We say a smooth scalar, vector or matrix valued function  $\sigma(\xi)$ to be a classical symbol of order $m$  provided
$$
|\d_\xi^\a \sigma(\xi)|\leq C_\a \lxi^{m-\a},\quad \lxi:= \left(1+|\xi|^2\right)^{\frac{1}{2}},\quad \mbox{for any $\a\in \N^d$}.
$$
We use $S^m$ to denote the set of all classical symbols of order $m$. The classical Fourier multiplier associated with a symbol $\sigma(\xi)$ is denoted by $\sigma(D_x)$, and is defined as
\be\label{def-Fourier1}
\sigma( D_x) u:=\mathfrak{F}^{-1}[\sigma(\xi)\hat u (\xi)]=\mathfrak{F}^{-1}[\sigma]*u,
\ee
where $\hat u(\xi)£º=\mathfrak{F}[u](\xi)$ is the Fourier transform of $u$ and $\mathfrak{F}^{-1}$ denotes the inverse of Fourier transform.

The semiclassical Fourier multiplier associated with a symbol $\sigma(\xi)$ is denoted by $\sigma(\e D_x)$, and is defined as
\be\label{def-Fourier2}
\sigma(\e D_x) u:=\mathfrak{F}^{-1}[\sigma(\e\xi)\hat u (\xi)]=\mathfrak{F}^{-1}[\sigma(\e\cdot)]*u=\e^{-d}\mathfrak{F}^{-1}[\sigma]\left(\frac{\cdot}{\e}\right)*u.
\ee

The definitions in \eqref{def-Fourier1} and \eqref{def-Fourier2} can be generated to less regular symbols $\s$ as long as the definitions make sense.

 We now give two properties that we will use in this paper for classical and semiclassical Fourier multipliers. The first one is rather direct:
\begin{lem}\label{lem-bound} Let $\sigma\in L^\infty$, then for any $s\in \R$ and $\e>0$:
\be\label{bound-fourier1}
\|\sigma(D_x) u\|_{H^s} \leq  \|\sigma(\cdot)\|_{L^\infty} \| u\|_{H^{s}}, \quad \|\sigma(\e D_x) u\|_{H^s} \leq  \|\sigma(\cdot)\|_{L^\infty} \| u\|_{H^{s}}.\nn
\ee
\end{lem}

\smallskip

The second one is about the commutator estimates.
\begin{lem}\label{lem-comm} Let $\sigma\in C^1$ such that $\|\nabla_\xi \sigma\|_{L^\infty}<\infty $ and  $g(x)\in H^{d/2+1+\eta_0}$  a scalar function for some $\eta_0>0$. Then there holds for any $s\geq 0$:
\be\label{bound-fourier2}
\|[\sigma(\e D_x),g(x)] u \|_{H^s} \leq \e\, C_{\eta_0}\, 2^s \|\nabla_\xi\sigma\|_{L^\infty} \left(\| g\|_{H^{\frac{d}{2}+1+\eta_0}} \|u\|_{H^{s}}+\|g\|_{H^{s+1}} \|u\|_{H^{\frac{d}{2}+\eta_0}}\right).\nn
\ee
\end{lem}
The point of Lemma \ref{lem-comm} is that the commutator of a semiclassical Fourier multiplier and a regular scalar function is  of order $\e$.

\begin{proof}[Proof of Lemma \ref{lem-comm}] Let
$$I(\xi):=\mathfrak{F}\big[[\sigma(\e D_x),g(x)] u\big ](\xi).$$
Then
$$
\|[\sigma(\e D_x),g(x)] u \|_{H^s}=\|\lxi^s I(\xi) \|_{L^2}.
$$
By the definition of semiclassical Fourier multiplier, we have
\ba\label{compute-commutator1}
&I(\xi)=\mathfrak{F}[\s(\e D_x)(gu)]-\mathfrak{F}[g\s(\e D_x)(u)]=\s(\e \xi)\mathfrak{F}[(gu)]-\mathfrak{F}[g\s(\e D_x)(u)]\\
&\qquad =\s(\e \xi) (\hat g*\hat u)(\xi)-\big(\hat g * (\sigma(\e \cdot) \hat u\big)(\xi)\\
&\qquad=\s (\e \xi)\int_{\R^d} \hat g(\eta) \hat u(\xi-\eta)d\eta-\int_{\R^d} \hat g(\eta) \s (\e \xi-\e\eta)\hat u(\xi-\eta)d\eta\\
&\qquad=\int_{\R^d} \hat g(\eta)\left(\s (\e \xi)-\s (\e \xi-\e\eta)\right)  \hat u(\xi-\eta)d\eta\\
&\qquad=\int_{\R^d} \hat g(\eta) \int_0^1 \e\eta\cdot (\nabla_\xi\s) (\e \xi-\e(1- t) \eta)dt \, \hat u(\xi-\eta) \,d\eta.\nn
\ea
Then
\ba\label{compute-commutator2}
\left|\lxi^sI(\xi)\right| &\leq \e \|\nabla_\xi \s\|_{L^\infty} \int_{\R^d} \lxi^s |\eta||\hat g(\eta)|\,| \hat u(\xi-\eta)|\,d\eta\\
&\leq \e \|\nabla_\xi \s\|_{L^\infty} \Big(\int_{|\eta|>\frac{|\xi|}{2}} \lxi^s |\eta||\hat g(\eta)|\,| \hat u(\xi-\eta)|\,d\eta\\
&\qquad \qquad \qquad+\int_{|\eta|\leq \frac{|\xi|}{2}} \lxi^s |\eta||\hat g(\eta)|\,| \hat u(\xi-\eta)|\,d\eta \Big)\\
&\leq \e 2^s  \|\nabla_\xi \s\|_{L^\infty} \Big(\int_{|\eta|>\frac{|\xi|}{2}} \langle \eta\rangle^{s} |\eta||\hat g(\eta)|\,| \hat u(\xi-\eta)|\,d\eta\\
&\qquad \qquad \qquad+\int_{|\eta|\leq \frac{|\xi|}{2}} \langle\xi- \eta\rangle ^s |\eta||\hat g(\eta)|\,| \hat u(\xi-\eta)|\,d\eta \Big)\\
&\leq \e 2^s  \|\nabla_\xi \s\|_{L^\infty} \left(|\lxi^{s+1}\hat g(\xi)|*|\hat u(\xi)|+|\xi\hat g(\xi)|*|\lxi^{s}\hat u(\xi)|\right).
\nn\ea
Young's inequality yields
\ba\label{compute-commutator3}
\left|\lxi^sI(\xi)\right|_{L^2} \leq \e 2^s\|\nabla_\xi \s\|_{L^\infty}\left(\|\lxi^{s+1}\hat g(\xi)\|_{L^2}\|\hat u(\xi)\|_{L^1}+\|\xi\hat g(\xi)\|_{L^1}\|\lxi^{s}\hat u(\xi)\|_{L^2}\right).\nn
\ea
H\"older's inequality implies
$$
\|\xi\hat g(\xi)\|_{L^1}\leq C_{\eta_0} \|\lxi^{d/2+1+\eta_0} \hat g(\xi)\|_{L^2},\quad \|\hat u(\xi)\|_{L^1}\leq C_{\eta_0} \|\lxi^{d/2+\eta_0}\hat u(\xi)\|_{L^2}.
$$
Finally, we obtain
\ba\label{compute-commutator4}
\left|\lxi^sI(\xi)\right|_{L^2} \leq  \e C_{\eta_0} 2^s \|\nabla_\xi \s\|_{L^\infty}\left(\| g\|_{H^{\frac{d}{2}+1+\eta_0}} \|u\|_{H^{s}}+\|g\|_{H^{s+1}} \|u\|_{H^{\frac{d}{2}+\eta_0}}\right).\nn
\ea
This completes the proof of Lemma \ref{lem-comm}.
\end{proof}

\section{Perturbed system and diagonalization}\label{sec:stability}

Now we start proving Theorem \ref{thm:general}. From now on, we suppose Assumption \ref{ass-spec}, Assumption \ref{ass-app} and Assumption \ref{ass-ps-tran} are satisfied.

\subsection{Perturbed system near approximate solution}
Associated with the approximate solution $U_a$ given in Assumption \ref{ass-app}, we define the perturbation
\be\label{def-dot-U}
\dot U:=\frac{U-U_a}{\e},
\ee
where $U \in C\left([0,T^*_\e);H^s\right)$ is the local-in-time solution to original Cauchy problem \eqref{00}.
Then at least over time interval $[0,T^*_\e)$, the perturbation $\dot U$ solves
 \be \label{eq-dot-U} \left\{ \begin{aligned} &\d_t \dot U  + \frac{1}{\e} A(\d_x) \dot  U + \frac{1}{\e^2} A_0 \dot U = 2B(U_a)\dot U +\e B(\dot U,\dot U)+ \e R^\e,\\
 &\dot U(0)= \psi^\e,\end{aligned}\right.
 \ee
where the linear operator $B(U_a)$ is defined as
\be\label{def-B-W}
B(U_a) W:=B(U_a, W),\quad \mbox{for any $W\in \C^{N}$}.\nn
\ee
The remainder $(R^\e,\psi^\e)$ satisfies the uniform estimate given in \eqref{bound-R-Psi}.

\smallskip

To prove Theorem \ref{thm:general}, it is sufficient to show the existence and uniform estimates for the solution of \eqref{eq-dot-U} over corresponding long time intervals.


The perturbed system \eqref{eq-dot-U} has small nonlinearity of order $O(\e)$. By careful, rather classical  analysis ($L^2$ estimate and Grownwall's inequality), it can be shown that the maximal existence time, denoted by $\dot T^*_\e$,  to Cauchy problem \eqref{eq-dot-U} satisfies
$$
\lim_{\e \to 0} \dot T_\e^*=\infty.
$$
By employing the argument in \cite{CD_DS}, one can even show the existence up to time of the logarithmic order:
$$
\dot T_\e^* \geq T_0 |\ln \e|,\quad \mbox{ for some $T_0>0$ independent of $\e$}.
$$

To show the existence up to even longer time of order $O(1/\e^\g)$,  we need to discover more structure of the system  \eqref{eq-dot-U}. To this end, we will diagonalize the differential operator on the left-hand side of \eqref{eq-dot-U} by diagonalizing  the corresponding symbol, then consider the system mode by mode.

\subsection{Diagonalization}\label{sec:diag}

 According to the smooth spectral decomposition assumed in Assumption \ref{ass-spec}, we can write
\be\label{sp-A-op}
A(\e D_x)+A_0/i=\sum_{j=1}^J \l_j(\e D_x) \Pi_j(\e D_x),\quad D_x:=\d_x/i.\nn
\ee

We want to go deep to the structure of the system in \eqref{eq-dot-U}. Hence, we consider the system mode by mode, through the following change of unknown:
\be\label{def-dot-U1}
\dot U_1=\bp\dot U_1^1\\ \vdots \\\dot U_1^J\ep:=\bp \Pi_1 (\e D_x)\dot U\\ \vdots \\ \Pi_J(\e D_x) \dot U\ep \in \R^{JN}.
\ee
We remark that, by Lemma \ref{lem-bound}, $\Pi_j(\e D_x),~1\leq j\leq J$ are linear operators bounded from $H^s$ to $H^s$ for any $s\in \R$. Hence
$$
\|\dot U_1(t,\cdot)\|_{H^s}\leq C \|\dot U(t,\cdot)\|_{H^s},\quad \mbox{for any $s\in \R$ and any $t\geq 0$}.
$$
Conversely,  we can reconstruct $\dot U$ via $\dot U_1$:
$$
\dot U:=\sum_{j=1}^J U_1^j
$$
 due to the fact
$$
\sum_{j=1}^J\Pi_j = \Id.
$$
We observe that
\be\label{Ua-U0-Ur}
B(U_a)=B(U_0)+\e B(U_r),\quad U_r:=U_1+\e U_2+\e^2 U_3.
\ee
Then by \eqref{eq-dot-U}, the equation in $\dot U_1$ is of the form
\be\label{eq-dot-U1}
\d_t \dot U_1  + \frac{i}{\e^2} A_1(\e D_x) \dot  U_1 =B_1 \dot U_1 +\e B_r \dot U_1 +\e F_1(\dot U_1,\dot U_1)+ \e R_1.
\ee

\smallskip

The propagator $A_1$ on the left-hand side is a diagonal matrix valued semiclassical Fourier multiplier
\be\label{def-A1}
A_1(\e D_x):={\rm diag}\,\{\l_1(\e D_x), \cdots, \l_J(\e D_x)\}.\nn
\ee

\smallskip

The leading linear operator $B_1$ on the right-hand side is
\be\label{def-B1}
B_1:=2\bp \Pi_j(\e D_x ) B(U_0) \Pi_{j'}(\e D_x)\ep_{1\leq j,j'\leq J},
\ee
which is of matrix form and is associated with the leading term $U_0$. By the form of $U_0$ in Assumption \ref{ass-app}, we have
\be\label{def-int-coe}
\Pi_j(\e D_x ) B(U_0) \Pi_{j'}(\e D_x):=\sum_{p=\pm1} e^{-ip\o t/\e^2} \Pi_j(\e D_x ) B(U_{0,p}) \Pi_{j'}(\e D_x).
\ee
The terms $\Pi_j(\e D_x ) B(U_{0,p}) \Pi_{j'}(\e D_x)$ are also named {\em interaction coefficients}. To specify, $\Pi_j(\e D_x ) B(U_{0,p}) \Pi_{j'}(\e D_x)$ is called {\em $(j,j',p)$ interaction coefficient}.

\smallskip

The remainder linear operator $B_r$ is
\be\label{def-P1}
B_r:=2\bp \Pi_j(\e D_x ) B(U_r) \Pi_{j'}(\e D_x)\ep_{1\leq j,j'\leq J},\nn
\ee
which is associated with the remainder term $U_r$ defined in  \eqref{Ua-U0-Ur}.

\smallskip

The nonlinear term $F_1$ is
\be\label{def-F1}
F_1(\dot U_1,\dot U_1):=\bp \Pi_1(\e D_x ) B(\dot U,\dot U)\\ \vdots \\\Pi_J(\e D_x ) B(\dot U,\dot U)\ep,\quad \dot U=\sum_{j=1}^J \dot U_1^j.\nn
\ee

\smallskip

Finally the remainder $R_1$ is

\be\label{def-R1}
R_1:=\bp \Pi_1(\e D_x ) R^\e\\ \vdots \\\Pi_J(\e D_x ) R^\e \ep.\nn
\ee

\medskip

To avoid notational complexity, we rewrite \eqref{eq-dot-U1} in the following more compact form
\be\label{eq-dot-U1-1}
\d_t \dot U_1  + \frac{i}{\e^2} A_1(\e D_x) \dot  U_1 =B_1 \dot U_1 +\e \mathcal{R}_1,
\ee
where $\mathcal{R}_1$ is the sum of all the $O(\e)$ terms.  By the uniform estimates for the approximate solution assumed in Assumption \ref{ass-app}, and by Lemma \ref{lem-bound} and Lemma \ref{lem-comm} about the actions of Fourier multipliers,  we have the estimate
\be\label{est-R1}
\|{\cal R}_{1}(t,\cdot)\|_{H^\mu} \leq C \left(1+\|\dot U (t,\cdot)\|_{L^\infty}\right)\|\dot U(t,\cdot)\|_{H^\mu},\quad \mbox{for all $0\leq \mu \leq s$}.\nn
\ee
The initial datum of $\dot U_1$ is
\be\label{ini-dotU1}
\dot U_1(0)=\bp \Pi_1 (\e D_x)\psi^\e\\ \vdots \\ \Pi_J(\e D_x) \psi^\e\ep,
\ee
which is uniformly bounded in $H^s$.

\section{Long time stability: Part I}\label{sec:pf-1st}

This section is devoted to proving the first part of Theorem \ref{thm:general}, that is the stability over time of order $O(1/\e^{\a})$ under Assumption \ref{ass-spec}, Assumption \ref{ass-app} and Assumption \ref{ass-ps-tran}.

\subsection{A singular localization}\label{sec:sin-loc}

To show long time of order $O(1/\e^\a)$ well-posedness for \eqref{eq-dot-U1-1} with $O(1)$ initial datum \eqref{ini-dotU1}, the idea here is to eliminate the $O(1)$ term $B_1$ on the right-hand side of \eqref{eq-dot-U1-1} up to a small remainder of order $O(\e^\a)$ in this section. Then we employ the classical theory to obtain the long time existence. However, the strong transparency condition is not satisfied in our setting, so we cannot simply use the normal form reduction method to achieve this. The main novelty of our study is to carry out a singular localization on the interaction coefficients; together with the normal form reduction, we show that we can eliminate the $O(1)$ interaction coefficients up to small remainders.

\medskip

We recall the definition of resonance sets for any $1\leq j\leq J, \ 1\leq j'\leq J,\ p\in \{-1,1\}$ in Assumption \ref{ass-ps-tran}:
\be\label{def-re-set}
R_{j,j',p}:=\{\xi\in\R^d\,:\,\l_j(\xi)-\l_{j'}(\xi)-p \o=0\}.\nn
\ee
Compared to the definition in \eqref{def-Rjjp}, we remark that here we have the zero spatial wave number: $k=0$.

If for some $(j',j,p)$ the corresponding resonance set $R_{j,j',p}$ is empty, by Assumption \ref{ass-ps-tran} and the smoothness and boundedness of $\Pi_j(\cdot)$, the following strong transparency condition is automatically  satisfied:
\be\label{no-reso2}
\big|\Pi_j(\xi)B(e_p)\Pi_{j'} (\xi) \big|\leq C |\l_j(\xi)-\l_{j'}(\xi)-p\o|.\nn
\ee
This indicates that, if  $(j,j',p) \in J_r$ defined as
\be\label{def:J0}
J_r:=\{(j,j',p)\,:\, R_{j,j',p}=\emptyset\},\nn
\ee
the exponent $\a_{j,j',p}$ in \eqref{sttrans-new} can be taken to be $1$. We thus introduce the index set
\be\label{def:J1}
J_1:=\{(j,j',p)\,:\, \mbox{\eqref{sttrans-new} holds for $\a_{j,j',p}=1$}\}\supset J_r.
\ee

Now we introduce smooth cut-off functions $\chi_{j,j',p}$ that are supported near resonance sets:
\ba\label{def-cut-sing}
 &\mbox{If $(j,j',p)\in J_1$}, \quad \chi_{j,j',p}\equiv 0.\\
 &\mbox{If $(j,j',p)\not\in J_1$}, \quad \chi_{j,j',p}\in C_c^\infty\left( R_{j,j',p}^{2h_\e}\right),\quad \chi_{j,j',p}\equiv 1 \ \mbox{on}\ R_{j,j',p}^{h_\e},\quad 0\leq \chi_{j,j',p}\leq 1,
\ea
where $0<h_\e<1$ is a small positive number depending on $\e$ and is to be determined later on, and
\be\label{def-Rh}
R_{j,j',p}^{h}:=\{\xi\in\R^d\,:\,|\l_j(\xi)-\l_{j'}(\xi)-p \o|< h \},\quad \mbox{for $h>0$ small}.
\ee
By the compactness assumption on resonance sets in Assumption \ref{ass-ps-tran}, for sufficient small $h$, the sets $R_{j,j',p}^{h}$ defined in \eqref{def-Rh} are uniformly bounded. The index $h_\e$ will be chosen relatively small such that $ R_{j,j',p}^{2h_\e}$ are all bounded.

We consider the decomposition of the leading linear operator $B_1$  defined in \eqref{def-B1} and \eqref{def-int-coe}:
\be\label{dec-B1-in-out}
B_1:=B_{in}+B_{out}
\ee
with
\ba\label{def-Bin-out}
& B_{in}:=2\sum_{p=\pm1} e^{-ip\o t/\e^2}\left( \chi_{j,j',p}(\e D_x) \Pi_j(\e D_x ) B(U_{0,p}) \Pi_{j'}(\e D_x)\right)_{1\leq j,j'\leq J},\\
& B_{out}:=2\sum_{p=\pm1} e^{-ip\o t/\e^2}\left( (1-\chi_{j,j',p})(\e D_x) \Pi_j(\e D_x ) B(U_{0,p}) \Pi_{j'}(\e D_x)\right)_{1\leq j,j'\leq J}.\\
\ea

The part $B_{in}$ is localized near the resonances while the other part $B_{out}$ is localized away from the resonances. However, this localization depends on $\chi_{j,j',p}$ which may be singular in $\e$ duce to the definition in \eqref{def-cut-sing}. Indeed, by \eqref{def-cut-sing}, the support of $\chi_{j,j',p}$ shrinks to the resonance set $R_{j,j',p}$ if $h_\e\to 0$ as $\e\to 0$. By our choice of $h_\e$ later on (see \eqref{choi-he1} and \eqref{choi-he20}), we do have $h_\e\to 0$ as $\e\to 0$. This causes the derivatives of $\chi_{j,j',p}$ could be unbounded as $\e \to 0$.  This is why we call this localization to be singular.

\medskip

First of all, we show that under Assumption \ref{ass-ps-tran}, the part $B_{in}$ near the resonance is small of order $O(h_\e^\a)$:

\begin{prop}\label{prop:Bin}
There exits $C>0$ such that for any  $d/2<\mu\leq s$ and any $V\in {H^\mu}$,  there holds
\be\label{Bin-small}
\|B_{in} V\|_{H^\mu}\leq C\,(h_\e^\a+\e) \, \|V\|_{H^\mu}.\nn
\ee
\end{prop}

\begin{proof}[Proof of Proposition \ref{prop:Bin}] By the definition of $\chi_{j,j',p}$  in \eqref{def-cut-sing}, it is sufficient to prove for any $(j,j',p)\not\in J_1$ and any $u\in H^{\mu}$, there holds
\be\label{Bin-small-1}
\| \chi_{j,j',p}(\e D_x) \Pi_j(\e D_x ) B(U_{0,p}) \Pi_{j'}(\e D_x) u\|_{H^\mu}\leq C\,(h_\e^\a+\e) \, \|u\|_{H^\mu}.
\ee

By using \eqref{ass-u0} and \eqref{ass-U01} in Assumption \ref{ass-app} and the actions of semiclassical Fourier multiplier in Lemma \ref{lem-bound},   we compute
\ba\label{Bin-small-2}
&\| \chi_{j,j',p}(\e D_x) \Pi_j(\e D_x ) B(U_{0,p}) \Pi_{j'}(\e D_x) u\|_{H^\mu}\\
&\quad = \| \chi_{j,j',p}(\e D_x) \Pi_j(\e D_x ) B(g_p e_p) \Pi_{j'}(\e D_x) u\|_{H^\mu}\\
&\quad \leq \| \chi_{j,j',p}(\e D_x) \Pi_j(\e D_x ) B(e_p) \Pi_{j'}(\e D_x) (g_p u)\|_{H^\mu}\\
&\qquad +\| \chi_{j,j',p}(\e D_x) \Pi_j(\e D_x ) B(e_p) [g_p, \Pi_{j'}(\e D_x)] u\|_{H^\mu}\\
&\quad \leq  \|\chi_{j,j',p}(\xi)\Pi_j(\xi) B(e_p) \Pi_{j'}(\xi))\|_{L^\infty_\xi} \|g_pu\|_{H^\mu}\\
&\qquad +\|\chi_{j,j',p}(\xi) \Pi_j(\xi ) B(e_p)\|_{L^\infty_\xi} \|[g_p, \Pi_{j'}(\e D_x)] u\|_{H^\mu}.\\
\ea

By the definition of $\chi_{j,j',p}$ in \eqref{def-cut-sing} and the condition \eqref{sttrans-new} in Assumption \ref{ass-app}, we have
\be\label{Bin-small-3}
\|\chi_{j,j',p}(\xi)\Pi_j(\xi) B(e_p) \Pi_{j'}(\xi))\|_{L^\infty_\xi}\leq C\,h_\e^{\a},\quad \|\chi_{j,j',p}(\xi) \Pi_j(\xi ) B(e_p)\|_{L^\infty_\xi} \leq C.\nn
\ee

By the regularity assumption $g_p\in L^\infty (0,\infty;H^{s+1})$ in Assumption \ref{ass-app}, there holds
\be\label{Bin-small-4}
\|g_pu\|_{H^\mu}\leq \| g_p\|_{H^\mu}\|u\|_{H^\mu} \leq C\,\|u\|_{H^\mu}.\nn
\ee

By Lemma \ref{lem-comm} concerning the commutator estimate, we have
\be\label{Bin-small-5}
\|[g_p, \Pi_{j'}(\e D_x)] u\|_{H^\mu}\leq C\,\e\,\|g_p\|_{H^{\mu+1}}\| u\|_{H^\mu}\leq C\,\e\,\|u\|_{H^\mu}.
\ee

Combining the estimates in \eqref{Bin-small-2}-\eqref{Bin-small-5} implies \eqref{Bin-small-1}. The proof is completed.

\end{proof}

Now it is left to deal with the part localized away from resonance sets.

\subsection{Normal form reduction}

Since $B_{out}$ is localized away from resonance, we can employ the normal reduction method to eliminate it up to some remainder. The issue is that due the singularity of the localization functions $\chi_{j,j',p}$ in $\e$, the remainder may not be small.

We will see later on, we can choose $h_\e$ properly to achieve a small remainder. We need to also choose $h_\e$  such that the remainder $B_{out}$ after the normal form reduction is of the same order as $B_{in}$ obtained in Proposition \ref{prop:Bin} in order to obtain the minimum remainder.

\medskip

We introduce the following formal change of unknown
\be\label{change1}
\dot U_2=\left(\Id +\e^2 M\right)^{-1} \dot U_1,
\ee
for some operator $M$ of the form
\be\label{M-form}
M=\sum_{p=\pm1}e^{-ip\o t/\e^2}\left( M_{jj'}^{(p)}\right)_{1\leq j,j'\leq J}
\ee
with the operator elements $M_{jj'}^{(p)}$ to be determined.

Then, by \eqref{eq-dot-U1-1}, the system in $\dot U_2$ has the form
\ba\label{eq-dot-U2}
&\d_t \dot U_2  + \frac{i}{\e^2} A_1(\e D_x) \dot  U_2 =\left(\Id +\e^2 M\right)^{-1}\left(B_{out}-i[A_1(\e D_x),M]-\e^2 \d_t M\right) \dot U_2 \\
&\qquad+\left(\Id +\e^2 M\right)^{-1}\left(\e^2 B_{out} M \dot U_2+ B_{in}\left(\Id +\e^2 M\right) \dot U_2 +\e \mathcal{R}_1\right).
\ea
The idea is to find some operator $M$ properly such that the $O(1)$ term on the right-hand side of \eqref{eq-dot-U2} is eliminated with a small remainder. This is done in the following proposition.
\begin{prop}\label{prop:M}
There exist symbols $\widetilde M_{jj'}^{(p)}(\xi) \in S^0$, $1\leq j,j'\leq J,~p\in \{-1,1\}$, such that $M$ defined in \eqref{M-form} with $M_{jj'}^{(p)}:=\widetilde M_{jj'}^{(p)}(\e D_x)\circ g_{p}$, which denotes the composition of Fourier multiplier $\widetilde M_{jj'}^{(p)}(\e D_x)$ and function multiplier $g_p$, satisfies
\be\label{prop-eq-M}
B_{out}-i[A_1(\e D_x),M]-\e^2 \d_t M=(\e \,h_\e^{\a-1}+\e^2 \,h_\e^{\a-1}+\e) M_r,\nn
\ee
where $M_r$ is a linear operator satisfying the estimate:
\be\label{est-Mr0}
\|M_r V\|_{H^\mu}\leq C \,\|V\|_{H^\mu},\quad \mbox{for any  $d/2<\mu\leq s$ and any $V\in {H^\mu}$}.
\ee
\end{prop}

\begin{proof} [Proof of Proposition \ref{prop:M}]
Given $M$ of the form \eqref{M-form}, we compute
\ba\label{A-M-com}
&[A_1(\e D_x),M]= A_1(\e D_x) M-M A_1(\e D_x)\\
&=\sum_{p=\pm1}e^{-ip\o t/\e^2}\left( \l_j M_{jj'}^{(p)}- M_{jj'}^{(p)} \l_{j'}\right)_{1\leq j,j'\leq J}\\
&=\sum_{p=\pm1}e^{-ip\o t/\e^2}\left((\l_j-\l_{j'}) M_{jj'}^{(p)}\right)_{1\leq j,j'\leq J} + M_r^{(1)},\nn
\ea
where we used the simplified notation $\l_j:=\l_j(\e D_x),~j\in\{1,\cdots, J\}$ and
\be\label{def-Mr1}
M_r^{(1)}:=\sum_{p=\pm1}e^{-ip\o t/\e^2}\left([\l_{j'},M_{jj'}^{(p)}]\right)_{1\leq j,j'\leq J}.
\ee

We then compute
\be\label{dt-M}
\e^2 \d_t M =\sum_{p=\pm1}e^{-ip\o t/\e^2} (-ip\o )\left( M_{jj'}^{(p)}\right)_{1\leq j,j'\leq J}+ M_r^{(2)},\nn
\ee
where
\be\label{def-Mr2}
M_r^{(2)} :=\e^2 \sum_{p=\pm1}e^{-ip\o t/\e^2} \left( \d_t M_{jj'}^{(p)}\right)_{1\leq j,j'\leq J}.
\ee

Then
\ba\label{A-M-com-dt-M}
&i[A_1(\e D_x),M]+\e^2 \d_t M\\
&=\sum_{p=\pm1}e^{-ip\o t/\e^2}\left(i(\l_j-\l_{j'}-p\o) M_{jj'}^{(p)}\right)_{1\leq j,j'\leq J} + M_r^{(1)} + M_r^{(2)}.\nn
\ea

By \eqref{def-Bin-out} and Assumption \ref{ass-app}, we have
\ba\label{Bout-cal1}
&B_{out}=2\sum_{p=\pm1} e^{-ip\o t/\e^2}\left( (1-\chi_{j,j',p})(\e D_x) \Pi_j(\e D_x ) B(U_{0,p}) \Pi_{j'}(\e D_x)\right)_{1\leq j,j'\leq J}\\
&=2\sum_{p=\pm1} e^{-ip\o t/\e^2}\left( (1-\chi_{j,j',p})(\e D_x) \Pi_j(\e D_x ) B(e_p) \Pi_{j'}(\e D_x) \circ g_p\right)_{1\leq j,j'\leq J}+ M_r^{(3)}\nn
\ea
with
\be\label{def-Mr3}
 M_r^{(3)}:=2\sum_{p=\pm 1} e^{-ip\o t/\e^2} (1-\chi_{j,j',p})(\e D_x) \Pi_j(\e D_x) B(e_p) [g_p,\Pi_{j'}(\e D_x)].
\ee

Now we are ready to give the definitions of $\widetilde M_{jj'}^{(p)}(\xi)$:
\be\label{tilde-Mijp}
\widetilde M_{jj'}^{(p)}(\xi):=-2i(\l_j(\xi)-\l_{j'}(\xi)-p\o)^{-1} (1-\chi_{j,j',p})(\xi) \Pi_j(\xi) B(e_p) \Pi_{j'}(\xi).
\ee
We observe that such $\widetilde M_{jj'}^{(p)}(\xi)$ are well defined due to the localization away from resonances (see \eqref{def-cut-sing} and \eqref{def-Rh}). Moreover,  by the condition \eqref{sttrans-new} in Assumption \ref{ass-ps-tran} and the definition of the cut-off functions in \eqref{def-cut-sing}, we have
\be\label{est-tilde-Mijp}
\|\widetilde M_{jj'}^{(p)}(\xi)\|_{L^\infty_\xi}\leq C\, h_\e^{\a_{j,j',p}-1}.
\ee

Then for the operator $M$ defined in \eqref{M-form} with $M_{jj'}^{(p)}=\widetilde M_{jj'}^{(p)}(\e D_x) \circ g_{p}$, by Assumption \ref{ass-app} ($g_p\in C_b \left([0,\infty); H^{s+1}\right)\cap C_b^1 \left([0,\infty); H^{s}\right)$), we first have for any $d/2<\mu\leq s$ and any $u\in H^{\mu}$ that
\ba\label{est-M}
\left\|M u\right\|_{H^\mu} +\left\|(\d_t M) u\right\|_{H^\mu}\leq C\, \,h_\e^{\a-1}  \left\| u\right\|_{H^\mu}.
\ea

Moreover, direct computation gives
$$
B_1^t-i[A_1(\e D_x),M]-\e^2 \d_t M= \widetilde M_r
$$
with
$$
\widetilde M_r=- M_r^{(1)}-M_r^{(2)}+M_r^{(3)}.
$$
It is left to show the uniform bound for the operator $\widetilde M_r$.

Let $d/2<\mu\leq s$ and $u \in H^{\mu}$. We start estimating $M_r^{(1)}$.  By Lemma \ref{lem-bound} and Lemma \ref{lem-comm}, direct computation gives
\ba\label{est-Mr10}
&\left\|[\l_{j'},M_{jj'}^{(p)}] u\right\|_{H^\mu} =\left\|[\l_{j'}(\e D_x),\widetilde M_{jj'}^{(p)}(\e D_x) g_p]u\right\|_{H^\mu}\\
&\quad=\left\|\widetilde M_{jj'}^{(p)}(\e D_x)[\l_{j'}(\e D_x),g_p] u\right\|_{H^\mu}\\
&\quad\leq C\, h_\e^{\a-1}\left\|[\l_{j'}(\e D_x),g_p] u\right\|_{H^\mu}\\
&\quad\leq C\, \e\,h_\e^{\a-1} \|g_p\|_{L^\infty(0,\infty;H^{s+1})} \left\| u\right\|_{H^\mu}.
\ea
Then by the definition of $M_r^{(1)}$ in \eqref{def-Mr1}, we have
\ba\label{est-Mr1}
\left\|M_r^{(1)} u\right\|_{H^\mu} \leq C\, \e\,h_\e^{\a-1}  \left\| u\right\|_{H^\mu}.
\ea

Similarly, for $M_r^{(2)}$ and $M_r^{(2)}$ defined in \eqref{def-Mr2} and \eqref{def-Mr3}, by Lemma \ref{lem-bound} and Lemma \ref{lem-comm}, we deduce
\ba\label{est-Mr23}
&\left\|M_r^{(2)} u\right\|_{H^\mu} \leq C\, \e^2\,h_\e^{\a-1} \|\d_t g_p\|_{L^\infty(0,\infty;H^{s})} \left\| u\right\|_{H^\mu},\\
&\left\|M_r^{(3)} u\right\|_{H^\mu} \leq C\, \e\, \|g_p\|_{L^\infty(0,\infty;H^{s+1})} \left\| u\right\|_{H^\mu}.
\ea

Summing up the estimates in \eqref{est-Mr1} and \eqref{est-Mr23}, we obtain
\ba\label{est-Mr}
\left\|\widetilde M_r u\right\|_{H^\mu} \leq C\,( \e\,h_\e^{\a-1}+ \e^2\,h_\e^{\a-1}+\e) \left\| u\right\|_{H^\mu}.\nn
\ea

This completes the proof of Proposition \ref{prop:M}.

\end{proof}

\subsection{End of the proof}

In this section, we complete the proof of the first part of Theorem \ref{thm:general}. This is achieved by choosing $h_\e$ properly. First of all, we choose $h_\e$ such that
\be\label{choi-he10}
\e^2 h_\e^{\a-1} \to 0,\ \mbox{as $\e\to 0$}.
\ee
By \eqref{est-M}, we have for any $d/2<\mu\leq s$:
\ba\label{est-M1}
\left\|\e^2 M \right\|_{\mathcal{L}(H^\mu\to H^\mu)} \to 0,\ \mbox {as $\e\to 0$}.\nn
\ea
Then for $\e$ sufficient small,  the operator $\left(\Id +\e^2 M\right)$ is well defined and uniformly bounded from $H^\mu\to H^\mu$, and is invertible with a uniformly bounded inverse.

Thus, let $M$ be the operator determined in Proposition \ref{prop:M}, the change of variable \eqref{change1} is well defined. By Proposition \ref{prop:Bin} and Proposition \ref{prop:M}, the system \eqref{eq-dot-U2} in $\dot U_2$ becomes
\ba\label{eq-dot-U2-1}
&\d_t \dot U_2  + \frac{i}{\e^2} A_1(\e D_x) \dot  U_2 =(\e h_\e^{\a-1}+h_\e^{\a}+\e)\mathcal{R}_2,
\ea
where there hods the estimate for any $d/2<\mu\leq s$:
\be\label{est-R2}
\|\mathcal{R}_2(t,\cdot)\|_{H^\mu}\leq C \left(1+\|\dot U_2(t,\cdot)\|_{H^{\mu}}\right)\|\dot U_2(t,\cdot)\|_{H^{\mu}}.\nn
\ee

Finally $h_\e$ is chosen such that $\e h_\e^{\a-1}=h_\e^{\a}$ to achieve the smallest remainder. This is equivalent to
\be\label{choi-he1}
h_\e=\e,
\ee
which implies $\e h_\e^{\a-1}=h_\e^{\a}=\e^\a$. The condition \eqref{choi-he10} is also satisfied.

For the initial datum of $\dot U_2$, by \eqref{ini-dotU1} and \eqref{change1}, we have
\be\label{ini-dotU2}
\dot U_2(0)=(\Id +\e^2 M)^{-1} \bp \Pi_1 (\e D_x)\psi^\e\\ \vdots \\ \Pi_3(\e D_x) \psi^\e\ep
\ee
for which the $H^{s}$ norm is uniformly bounded in $\e$.

We consider another change of unknown corresponding to a rescalling in time:
\be\label{change3}
\dot U_3(t)=\dot U_2 (\e^{-\a}t).\nn
\ee
Then the equation and initial datum for $\dot U_3$ are
 \be \label{eq-dot-U3} \left\{ \begin{aligned} &\d_t \dot U_3  + \frac{i}{\e^{2+\a}} A_1(\e D_x) \dot  U_3 =\mathcal{R}_3,\\
 & \dot U_3(0)=\dot U_2(0),\end{aligned}\right.
 \ee
 where $\mathcal R_3(t):=(2+\e^{1-\a})\mathcal R_2(\e^{-\a}t)$ satisfies for any $d/2<\mu\leq s$:
  \be\label{est-R3}
\|\mathcal{R}_3(t,\cdot)\|_{H^\mu}\leq C (1+\|\dot U_3(t,\cdot)\|_{H^{\mu}})\|\dot U_3(t,\cdot)\|_{H^{\mu}}.\nn
\ee

Since $s>d/2$, then by the classical theory for the local-in-time well-posedness of symmetric hyperbolic systems (see for instance Chapter 2 of \cite{Ma} or Chapter 7 of \cite{Me}), there exists a unique local-in-time solution $\dot U_3\in L^\infty(0,T_1; H^{s})$ to Cauchy problem \eqref{eq-dot-U3} for some $T_1>0$ independent of $\e$.

Equivalently, there exists a unique solution $\dot U_2\in L^\infty(0,\frac{T_1}{\e^\a}; H^{s})$ to \eqref{eq-dot-U2-1}-\eqref{ini-dotU2}. We go back to $\dot U$ and obtain the well-posedness of \eqref{eq-dot-U} in  $L^\infty(0,\frac{T_1}{\e^\a}; H^{s})$. Since the approximate solution $U_a$ is globally well defined and uniformly bounded in $L^\infty(0,\infty;H^{s+1})$, we can reconstruct the solution $U$ for \eqref{00} in $L^\infty(0,\frac{T_1}{\e^{\a}}; H^{s})$ through \eqref{def-dot-U}. We then complete the proof for the first part of Theorem \ref{thm:general}.

 We now turn to prove the second part of Theorem \ref{thm:general}.

\section{Long time stability: Part II}\label{sec:pf-2nd}

This section is devoted to proving the second part of Theorem \ref{thm:general}, that is the stability of the approximate solution given in Assumption \ref{ass-app} over time $O(1/\e^{\a_1})$ with $\a_1=\min\{2\a,1\}$. We suppose Assumption \ref{ass-spec}, Assumption \ref{ass-app}, Assumption \ref{ass-ps-tran} and Assumption \ref{ass-ps-tran-add} are all satisfied.

 We will also employ the idea in Section \ref{sec:pf-1st}. The main idea of singular localization and normal form reduction is the same as in the proof of the first part in Section \ref{sec:pf-1st}. However, the analysis here is more delicate in order to achieve even longer time stability.

 There are two new key points. The first one is to define the singular decomposition $B_{in}$ and $B_{out}$, as well as the operator $M$ in the normal form reduction in such a way that we can avoid the commutator $M_{r}^{(1)}$ in \eqref{def-Mr1} which is of order $\e h_\e^{\a-1}$ (see \eqref{est-Mr10} and \eqref{est-Mr1}). This can be achieved by proper choice for the positions of $\chi_{j,j',p}$ in \eqref{B-jjp1}-\eqref{B-jjp3} and the positions of $g_p$ in \eqref{def-Qjj} which allows us to force the commutators in $\eqref{def-Mr1}$ to appear only associated with the constant eigenmode $\l_j$ or $\l_{j'}$ from Assumption \ref{ass-ps-tran-add}. The other remainders $M_{r}^{(2)}$ and $M_{r}^{(3)}$ are smaller of order $\e^2 h_\e^{\a-1}$ and $\e$ respectively (see \eqref{est-Mr23}) and we do not need to deal with them furthermore.

 The other key point is to avoid the commutators $[\chi_{j,j',p},g_p]$ which may be large because of the singularity of $\chi_{j,j',p}$ as $\e\to 0$. This can be also achieved by choosing the positions of cut-off functions $\chi_{j,j',p}$, $(1-\chi_{j,j',p})$ and scalar multiplier $g_p$ in the definitions of $B_{in}$, $B_{out}$ and the operator $M$ used in the normal form reduction, see \eqref{B-jjp1}-\eqref{B-jjp3} and \eqref{def-Qjj} later on.

\subsection{Refined singular localization}

The cut-off functions $\chi_{j,j',p}$ are the same as in Section \ref{sec:sin-loc}, while the definitions for the decomposition component $B_{in}$ and $B_{out}$ have to be modified.

For any $(j,j',p)$, we introduce the elements $B_{in}^{(j,j',p)}$ and $B_{out}^{(j,j',p)}$ in the following way:
\begin{itemize}
\item For any $(j,j',p)\in J_1$ which is defined in \eqref{def:J1}, we set
\be\label{B-jjp1}
B_{in}^{(j,j',p)}:=0,\quad B_{out}^{(j,j',p)}:=2\Pi_j(\e D_x ) B(U_{0,p}) \Pi_{j'}(\e D_x).
\ee

\item For any $(j,j',p)\not \in J_1$, by Assumption \ref{ass-ps-tran-add}, one of $\l_j(\xi)$ and $\l_{j'}(\xi)$ is constant.

If $\l_j(\xi)$ is constant, we set
\ba\label{B-jjp2}
& B_{in}^{(j,j',p)}:=2\Pi_j(\e D_x ) B(U_{0,p}) \Pi_{j'}(\e D_x)\chi_{j,j',p}(\e D_x),\\
& B_{out}^{(j,j',p)}:=2\Pi_j(\e D_x ) B(U_{0,p}) \Pi_{j'}(\e D_x)(1-\chi_{j,j',p})(\e D_x).
\ea

If $\l_{j'}(\xi)$ is constant, we set
\ba\label{B-jjp3}
& B_{in}^{(j,j',p)}:=2\chi_{j,j',p}(\e D_x)\Pi_j(\e D_x ) B(U_{0,p}) \Pi_{j'}(\e D_x),\\
& B_{out}^{(j,j',p)}:=2(1-\chi_{j,j',p})(\e D_x)\Pi_j(\e D_x ) B(U_{0,p}) \Pi_{j'}(\e D_x).
\ea
\end{itemize}

The new decomposition of $B_1:=\widetilde B_{in}+\widetilde B_{out}$ is defined as
\ba\label{def-Bin-out-new}
&\widetilde B_{in}:=\sum_{p=\pm 1}e^{-ip\o t/\e^2}\left(B_{in}^{(j,j',p)}\right)_{1\leq j,j'\leq J},\\
&\widetilde B_{out}:=\sum_{p=\pm 1}e^{-ip\o t/\e^2}\left(B_{out}^{(j,j',p)}\right)_{1\leq j,j'\leq J}.
\ea

We remark that, compared to $B_{in}$ and $B_{out}$ defined before in \eqref{def-Bin-out}, the new definition through \eqref{B-jjp1}-\eqref{def-Bin-out-new} pays more attention to the positions of the cut-offs $\chi_{j,j',p}$. We will see later on in the proof of Proposition \ref{prop:Q}, this choice of positions, together with the choice of positions of $g_p$ in \eqref{def-Qjj}, allows us to avoid relatively large commutators of order $\e h_\e^{\a-1}$ as well as the commutators $[\chi_{j,j',p},g_p]$.

First of all, similar to  Proposition \ref{prop:Bin}, we have:
\begin{prop}\label{prop:Bin-new}
Let $d/2<\mu\leq s$ and $V\in {H^\mu}$.  There holds
\be\label{Bin-small-new}
\|\widetilde B_{in} V\|_{H^\mu}\leq C\,(h_\e^\a+\e) \, \|V\|_{H^\mu}.\nn
\ee
\end{prop}

The proof of Proposition \ref{prop:Bin-new} is similar as that of Proposition \ref{prop:Bin}, that is to employ Lemma \ref{lem-bound} and Lemma \ref{lem-comm}, Assumption \ref{ass-app} and Assumption \ref{ass-ps-tran}, and the property of the cut-off functions $\chi_{j,j',p}$  in \eqref{def-cut-sing}. We do not repeat the details.

\subsection{Refined normal form reduction}

We employ the normal reduction method to deal with $\widetilde B_{out}$ which is localized away from resonance. Introduce the  change of variable
\be\label{NF-change-new}
\dot U_4=\left(\Id +\e^2 Q\right)^{-1} \dot U_1,
\ee
where $U_1$ is given in \eqref{def-dot-U1} and solves \eqref{eq-dot-U1-1},  $Q$ is an operator of the form
\be\label{Q-form}
Q=\sum_{p=\pm1}e^{-ip\o t/\e^2}\left( Q_{jj'}^{(p)}\right)_{1\leq j,j'\leq J}
\ee
with the operator elements $Q_{jj'}^{(p)}$ to be determined. Then, by \eqref{eq-dot-U1-1}, the system in $\dot U_4$ has the form
\ba\label{eq-dot-U4}
&\d_t \dot U_4  + \frac{i}{\e^2} A_1(\e D_x) \dot  U_4 =\left(\Id +\e^2 Q\right)^{-1}\left(\widetilde B_{out}-i[A_1(\e D_x),Q]-\e^2 \d_t Q\right) \dot U_4 \\
&\qquad+\left(\Id +\e^2 Q\right)^{-1}\left(\e^2 \widetilde B_{out} Q \dot U_4+ \widetilde B_{in}\left(\Id +\e^2 Q\right) \dot U_4 +\e \mathcal{R}_1\right).
\ea
One key result is the following:
\begin{prop}\label{prop:Q}
Let $\widetilde M_{jj'}^{(p)}(\xi)$ be the symbols defined in \eqref{tilde-Mijp}. Then the operator $Q$ defined in \eqref{Q-form} with
\ba\label{def-Qjj}
&Q_{jj'}^{(p)}:=g_{p}\circ \widetilde M_{jj'}^{(p)}(\e D_x) \ {\rm or}\  \widetilde M_{jj'}^{(p)}(\e D_x)\circ g_{p},\quad \mbox{if $(j,j',p)\in J_1$},\\
&Q_{jj'}^{(p)}:=g_{p}\circ \widetilde M_{jj'}^{(p)}(\e D_x), \quad \mbox{if $(j,j',p)\not\in J_1$ and $\l_{j}(\cdot)$ is constant},\\
&Q_{jj'}^{(p)}:= \widetilde M_{jj'}^{(p)}(\e D_x)\circ g_{p}, \quad \mbox{if $(j,j',p)\not\in J_1$ and $\l_{j'}(\cdot)$ is constant}
\ea
 satisfies
\be\label{prop-eq-Q}
\widetilde B_{out}-i[A_1(\e D_x),Q]-\e^2 \d_t Q=(\e^2 \,h_\e^{\a-1}+\e) Q_r,\nn
\ee
where  $Q_r$ is a linear operator satisfying the estimate:
\be\label{est-Qr}
 \|Q_r V\|_{H^\mu}\leq C \,\|V\|_{H^\mu},\quad \mbox{for any $d/2<\mu\leq s$ and $V\in H^{\mu}$}.
\ee
\end{prop}

\begin{rem}
We observe that, compared to Proposition \ref{prop:M}, the remainder estimate in \eqref{est-Qr} is better than that in \eqref{est-Mr0} where the term $\e h_\e^{\a-1}$ is eliminated.
\end{rem}
\begin{proof} [Proof of Proposition \ref{prop:Q}]
First of all, the operator $Q$ given in Proposition \ref{prop:Q} is well defined and satisfies for any $d/2<\mu\leq s$ and $V\in H^{\mu}$:
\be\label{est-Q}
 \|Q V\|_{H^\mu}+ \|(\d_t Q) V\|_{H^\mu} \leq C \,h_\e^{\a-1}\,\|V\|_{H^\mu}.
\ee

Given $Q$ of the form \eqref{Q-form}, we compute
\ba\label{A-Q-com}
&[A_1(\e D_x),Q]= A_1(\e D_x) Q - Q A_1(\e D_x)\\
&=\sum_{p=\pm1}e^{-ip\o t/\e^2}\left( \l_j Q_{jj'}^{(p)}- Q_{jj'}^{(p)} \l_{j'}\right)_{1\leq j,j'\leq J},\nn
\ea
where we used the simplified notation $\l_j:=\l_j(\e D_x),~j\in\{1,\cdots, J\}$.

We then compute
\be\label{dt-Q}
\e^2 \d_t Q =\sum_{p=\pm1}e^{-ip\o t/\e^2} (-ip\o )\left( Q_{jj'}^{(p)}\right)_{1\leq j,j'\leq J}+ Q_{r,1},\nn
\ee
where
\be\label{def-Qr1}
Q_{r,1} :=\e^2 \sum_{p=\pm1}e^{-ip\o t/\e^2} \left( \d_t Q_{jj'}^{(p)}\right)_{1\leq j,j'\leq J}.\nn
\ee
Direct computation gives that for any $d/2<\mu\leq s$ and $V\in H^{\mu}$:
\be\label{est-Qr1}
 \|Q_{r,1} V\|_{H^\mu} \leq C \,\e^2 h_\e^{\a-1}\,\|V\|_{H^\mu}.
\ee

Then
\ba\label{A-Q-com-dt-Q}
\widetilde B_{out}-i[A_1(\e D_x),Q]-\e^2 \d_t Q=\sum_{p=\pm1}e^{-ip\o t/\e^2}\left( Q_{j,j',p}^{(r)} \right)_{1\leq j,j'\leq J} - Q_{r,1},\nn
\ea
where
\be\label{def-Qrjjp}
Q_{j,j',p}^{(r)}:= B_{out}^{(j,j',p)}-i \left( \l_j Q_{jj'}^{(p)}- Q_{jj'}^{(p)}  \l_{j'}-p\o Q_{jj'}^{(p)}\right).\nn
\ee
We now estimate $Q_{j,j',p}^{(r)}$ case by case:

\begin{itemize}

\item For $(j,j',p)\in J_1$, we choose $Q_{jj'}^{(p)}=g_{p}\circ \widetilde M_{jj'}^{(p)}(\e D_x)$ (the same result follows if we choose $Q_{jj'}^{(p)}= \widetilde M_{jj'}^{(p)}(\e D_x)\circ g_{p}$), then
\ba\label{est-Qrjjp1}
Q_{j,j',p}^{(r)}&=2\Pi_j(\e D_x ) B(g_p e_p) \Pi_{j'}(\e D_x) -2 g_p \Pi_j(\e D_x ) B(e_p) \Pi_{j'}(\e D_x)+Q_{r,2}^{j,j',p}\\
&=Q_{r,2}^{j,j',p}+Q_{r,3}^{j,j',p},\nn
\ea
where
\be\label{def-Qr23jjp}
Q_{r,2}^{j,j',p} := -2 [\l_j,g_p]\widetilde M_{jj'}^{(p)}(\e D_x),\quad Q_{r,3}^{j,j',p} :=2[\Pi_j(\e D_x ),g_p] B( e_p) \Pi_{j'}(\e D_x).\nn
\ee
For any $(j,j',p)\in J_1$, $\a_{j,j',p}=1$. Then by \eqref{est-tilde-Mijp} and the estimates in Lemma \ref{lem-bound} and \eqref{lem-comm}, we have for any $d/2<\mu\leq s$ and $u\in H^{\mu}$:
\be\label{est-Qr23jjp}
 \|Q_{r,2}^{j,j',p} u\|_{H^\mu} + \|Q_{r,3}^{j,j',p} u\|_{H^\mu}\leq C \e\,\|u\|_{H^\mu}.
\ee

\item For $(j,j',p)\not \in J_1$ with $\l_j(\cdot)$ constant, we have
\ba\label{est-Qrjjp2}
Q_{j,j',p}^{(r)}&=2\Pi_j(\e D_x ) B(g_p e_p) \Pi_{j'}(\e D_x)(1-\chi_{j,j',p})(\e D_x) \\
&\quad-2 g_p \Pi_j(\e D_x ) B(e_p) \Pi_{j'}(\e D_x) (1-\chi_{j,j',p})(\e D_x) +Q_{r,4}^{j,j',p}\\
&=Q_{r,4}^{j,j',p}+Q_{r,5}^{j,j',p},\nn
\ea
where
\ba\label{def-Qr45jjp}
& Q_{r,4}^{j,j',p} := -2 [\l_j,g_p]\widetilde M_{jj'}^{(p)}(\e D_x)=0\\
& Q_{r,5}^{j,j',p} :=2[\Pi_j(\e D_x ),g_p] B( e_p) \Pi_{j'}(\e D_x)(1-\chi_{j,j',p})(\e D_x),\nn
\ea
where we used the fact $\l_j$ is constant. Again by the estimates in Lemma \ref{lem-bound} and Lemma \ref{lem-comm}, we have for any $d/2<\mu\leq s$ and $u\in H^{\mu}$:
\be\label{est-Qr45jjp}
 \|Q_{r,5}^{j,j',p} u\|_{H^\mu} \leq C \e\,\|u\|_{H^\mu}.
\ee

\item For $(j,j',p)\not \in J_1$ with $\l_{j'}(\cdot)$ constant, we have
\ba\label{est-Qrjjp3}
Q_{j,j',p}^{(r)}&=2(1-\chi_{j,j',p})(\e D_x) \Pi_j(\e D_x ) B(g_p e_p) \Pi_{j'}(\e D_x)\\
&\quad-2 (1-\chi_{j,j',p})(\e D_x)  \Pi_j(\e D_x ) B(e_p) \Pi_{j'}(\e D_x) g_p +Q_{r,6}^{j,j',p}\\
&=Q_{r,6}^{j,j',p}+Q_{r,7}^{j,j',p},\nn
\ea
where
\ba\label{def-Qr67jjp}
& Q_{r,6}^{j,j',p} := -2 \widetilde M_{jj'}^{(p)}(\e D_x) [\l_{j'},g_p]=0\\
& Q_{r,7}^{j,j',p} :=2(1-\chi_{j,j',p})(\e D_x) \Pi_{j}(\e D_x) B( e_p) [g_p,\Pi_{j'}(\e D_x)],\nn
\ea
where we used the fact $\l_{j'}$ is constant. Moreover, for any $d/2<\mu\leq s$ and $u\in H^{\mu}$ there holds
\be\label{est-Qr67jjp}
 \|Q_{r,7}^{j,j',p} u\|_{H^\mu} \leq C \e\,\|u\|_{H^\mu}.
\ee

\end{itemize}

By the estimates in \eqref{est-Qr1}, \eqref{est-Qr23jjp}, \eqref{est-Qr45jjp} and \eqref{est-Qr67jjp}, we conclude our result in Proposition \ref{prop:Q}.
\end{proof}

\subsection{End of the proof}

In this section, we complete the proof of the first part of Theorem \ref{thm:general}. Let $h_\e$ be such that
\be\label{choi-he20}
\e^2 h_\e^{\a-1} \to 0,\ \mbox{as $\e\to 0$}.
\ee

Then by \eqref{est-Q} for $\e$ sufficient small,  the operator $\left(\Id +\e^2 Q\right)$ as well as its inverse are uniformly bounded from $H^\mu\to H^\mu$. Thus, the change of variable \eqref{NF-change-new} is well defined. By Proposition \ref{prop:Bin-new} and Proposition \ref{prop:Q}, the system \eqref{eq-dot-U4} in $\dot U_4$ becomes
\ba\label{eq-dot-U4-1}
&\d_t \dot U_4  + \frac{i}{\e^2} A_1(\e D_x) \dot  U_4 =(\e^2 h_\e^{\a-1}+h_\e^{\a}+\e)\mathcal{R}_4,\nn
\ea
where there holds the estimate for any $d/2<\mu\leq s$:
\be\label{est-R4}
\|\mathcal{R}_4(t,\cdot)\|_{H^\mu}\leq C (1+\|\dot U_4(t,\cdot)\|_{H^{\mu}})\|\dot U_4(t,\cdot)\|_{H^{\mu}}.\nn
\ee

Finally $h_\e$ is chosen such that $\e^2 h_\e^{\a-1}=h_\e^{\a}$ to achieve the smallest remainder. This suggests
\be\label{choi-he2}
h_\e=\e^2.
\ee
This implies $\e^2 h_\e^{\a-1}=h_\e^{\a}=\e^{2\a}$. The condition \eqref{choi-he20} is also satisfied.

For the initial datum of $\dot U_4$, by \eqref{ini-dotU1} and \eqref{change1}, we have
\be\label{ini-dotU4}
\dot U_4(0)=(\Id +\e^2 Q)^{-1} \bp \Pi_1 (\e D_x)\psi^\e\\ \vdots \\ \Pi_3(\e D_x) \psi^\e\ep\nn
\ee
for which the $H^{s}$ norm is uniformly bounded in $\e$.

We consider another change of unknown corresponding to a rescalling in time:
\be\label{change45}
\dot U_5(t)=\dot U_4 (\e^{-\a_1}t),\quad \a_1:=\min\{2\a,1\}.\nn
\ee
Then the equation and initial datum for $\dot U_5$ are
 \be \label{eq-dot-U5} \left\{ \begin{aligned} &\d_t \dot U_5  + \frac{i}{\e^{2+\a_1}} A_1(\e D_x) \dot  U_5 =\mathcal{R}_5,\\
 & \dot U_5(0)=\dot U_4(0),\end{aligned}\right.
 \ee
 where $\mathcal R_5(t):=(2\,\e^{2\a-\a_1}+\e^{1-\a_1})\mathcal R_4(\e^{-\a_1}t)$ satisfies for any $d/2<\mu\leq s$:
  \be\label{est-R5}
\|\mathcal{R}_5(t,\cdot)\|_{H^\mu}\leq C \left(1+\|\dot U_5(t,\cdot)\|_{H^{\mu}}\right)\|\dot U_5(t,\cdot)\|_{H^{\mu}}.\nn
\ee

Then the classical theory gives the local-in-time existence and uniqueness of the solution $\dot U_5\in L^\infty(0,T_2; H^{s})$ to Cauchy problem \eqref{eq-dot-U5} for some $T_2>0$ independent of $\e$. Going back to $\dot U$ gives the well-posedness of \eqref{eq-dot-U} in  $L^\infty(0,\frac{T_2}{\e^{\a_1}}; H^{s})$. Since the approximate solution $U_a$ is globally well defined and uniformly bounded in $L^\infty(0,\infty;H^{s+1})$, we can reconstruct the solution $U$ for \eqref{00} in $L^\infty(0,\frac{T_2}{\e^{\a_1}}; H^{s})$ through \eqref{def-dot-U}. We then complete the proof for the second part of Theorem \ref{thm:general}.

\medskip

Now we have finished the proof of Theorem \ref{thm:general}. In the next section, we apply our result to the study of non-relativistic limit problems of Klein-Gordon equations.

\section{Example and application}\label{sec:motivation}

Our example contains the non-relativistic limit problems of Klein-Gordon equations. The Klein-Gordon equation is a relativistic version of the Schr\"odinger equation and is used to describe the  motion of a spinless particle with positive mass $m>0$. Let $c$ be the speed of light, $h$ be the Planck constant, then the typical form of the Klein-Gordon equation is
  \be\label{kg}
  \frac{h^2}{mc^2}\d_{tt} u-\frac{h^2}{m}\Delta u+m c^2 u =f(u),\quad t\geq 0,\quad x\in \R^d.\nn
  \ee
 Here $u=u(t,x)$ is a real-valued (or complex-valued) field, and $f(u)$ is a real-valued (or complex-valued) function.  By normalizing the mass such that $m=1$ and rescaling the time and space variables as
 $$
 \tilde u (t,x):=u(h^{-1}t,h^{-1}x),
 $$
 and by introducing $\e = c^{-1}$, we arrive at the following non-dimensional form of the Klein-Gordon equation
 \be\label{0}
\e^2\d_{tt} u-\Delta u+\frac{1}{\e^2}u =f(u),\quad t\geq 0,\quad x\in \R^d.
\ee
Here in \eqref{0}, we denote the new unknown $\tilde u$ still by the original notation $u$.

For fixed $\e$, the well-posedness of the Klein-Gordon equation is well studied (see for instance \cite{GV1, GV2}). Our concern is the \emph{long time} asymptotic behavior of the solution in the non-relativistic limit ($\e \to 0$)
with real initial data of the form
\be\label{ini-0}
 u(0)=u_{0,\e},\quad (\d_t u)(0)=\frac{1}{\e^2}u_{1,\e}.
\ee

The \emph{local-in-time} asymptotic behavior in the non-relativistic limit of \eqref{0}-\eqref{ini-0} is well studied both in mathematical analysis and in numerical computations, see for instance \cite{MN2, BD} and the recent result concerning higher order approximation by the authors in \cite{LZ1}. However, concerning the long time (for example of order $O(1/\e)$) asymptotic behavior in this setting, there are few results according to the authors' knowledge.

\subsection{Setting and main result}

With quadratic nonlinearity $f(u)=\l u^2,\ \l \in \R$, we will show that up to a change of unknowns, the Klein-Gordon equation \eqref{0} can be treated as a system of the form \eqref{00}. With additional regularity assumption on the initial data in \eqref{ini-0}, we verify that Assumption \ref{ass-spec}, Assumption \ref{ass-app}, Assumption \ref{ass-ps-tran} and Assumption \ref{ass-ps-tran-add} are all satisfied. Hence, we can apply Theorem \ref{thm:general} to obtain long time $O(1/\e)$ stability property.  Moreover, the leading term of the approximate solution solves linear Schr\"odinger equation. This shows rigourously that over long time of order $O(1/\e)$, the quadratic Klein-Gordon equation can be well approximated by the linear Schr\"odinger equation in the non-relativistic regime $\e \to 0$. However, this example model is rather non-physical since physical nonlinearities are of the form $f(u)= g(|u|^2) u$ which fulfills the gauge invariance. An extension of the theory presented in this paper to non quadratic nonlinearities is needed to consider such physical nonlinearities.

\medskip

We state our result.

\begin{theo}\label{thm:main}
Assume that the real initial datum $(u_{0,\e}, u_{1,\e})$   has the form
\be\label{ini-data0}
u_{0,\e}=\varphi_{0}+\e \varphi_\e,\quad u_{1,\e}=\psi_{0}+\e \psi_{\e}\nn
\ee
 with
\ba\label{reg-ini}
&(\varphi_0,\psi_0) \in (H^s)^2 \quad \mbox{independent of $\e$},\\
&\big\{(\varphi_\e,\psi_\e,\e\nabla \varphi_\e)\big\}_{0<\e<1} \mbox{~uniformly bounded in $(H^{s-4})^{d+2}$}
\ea
for  some $s>d/2+4$. Then there exists $\e_0>0$ such that for any $0<\e<\e_0$ the Cauchy problem  \eqref{0}-\eqref{ini-0} with $f(u)=\l\, u^2,~\l\in \R$ admits a unique solution $u\in L^\infty\left(0,\frac{T}{\e}; H^{s-4}\right)$
for some $T>0$ independent of $\e$. Moreover,  there exists a constant $C$ independent of $\e$ such that
\be\label{error-est10}
\left\|u-\left(e^{-it/\e^2}v +e^{it/\e^2} \bar v\right)\right\|_{L^\infty\left(0,\frac{T}{\e}; H^{s-4}\right)}\leq C\, \e,\nn
\ee
where $v\in C_b(0,\infty;H^s)\cap C_b^1(0,\infty;H^{s-2})$ is the solution to the following Cauchy problem associate with the linear Schr\"odinger equation
\be\label{g0-sch1}
 2iv_t+\Delta v=0,\quad v(0)=\frac{\varphi_0+i \psi_0}{2}.
 \ee

\end{theo}

The rest of this section is devoted to the proof of Theorem \ref{thm:main}.

\subsection{Reformulation of the equation}\label{sec:refo}

We rewrite the Klein-Gordon equation \eqref{0} as a symmetric hyperbolic system by introducing
\be\label{new-var}
U:=(w,v, u):=\left( \e\nabla^T u,\e^2\d_{t} u,u \right)^T:=\left(\e(\d_{x_1}u,\cdots, \d_{x_d} u),\e^2\d_t u,u\right)^T.\nn
\ee
Here the notation $\nabla:= (\d_{x_1},\cdots,\d_{x_d})^T$ is of column form. We remark that $0$ could be the scalar number zero, the zero column vector $0_d$, the zero row vector $0_d^T$ or the  zero matrix $0_{d\times d}$, but we will not specify if there is no confusion in the context.

Then the equation \eqref{0} is equivalent to
\be\label{000}
\d_t U+\frac{1}{\e}A(\d_x)U +\frac{1}{\e^2} A_0 U=F(U),
\ee
where
\be\label{def:AB}
A(\d_x):=-\bp 0 &\nabla & 0\\ \nabla^T &0&0\\0&0&0 \ep, \quad A_0:=\bp 0 &0 & 0\\ 0&0&1\\0&-1&0 \ep,\quad F(U)=\bp 0\\f(u)\\0\ep.
\ee

We consider in this paper the quadratic nonlinearity of the form $f(u)=\l\, u^2$ for some $\l >0$, we can write
\be\label{F-B}
F(U)=B(U,U)
\ee
with $B$ a symmetric bilinear form defined as
\be\label{def-B}\quad B(U_1,U_2)=-\l \bp 0\\u_1 u_2\\0\ep, \quad \mbox{for any }U_j=\bp w_j\\v_j\\u_j\ep, ~j\in\{1,2\}.\ee

Moreover, by  \eqref{ini-0} and the assumption in Theorem \ref{thm:main}, the initial datum is
\be\label{ini-data000}
U(0)=\left(\e \nabla^T u_{0,\e}, u_{1,\e}, u_{0,\e}\right)^T=\left(0, \psi_0, \phi_0 \right)^T+\e \left(\nabla^T (\phi_0+\phi_\e), \psi_{\e}, \phi_{\e}\right)^T.
\ee

Thus, we obtain a Cauchy problem \eqref{000}-\eqref{ini-data000} which has the form of \eqref{00}.

\subsection{Spectral decomposition}

We rewrite the linear differential operator on the left-hand side of \eqref{000} as
$$
\d_t +\frac{i}{\e^2}\left(A(\e D_x)+A_0/i\right),\quad D_x:=\d_x/i.
$$
The symbol of the semiclassical Fourier multiplier  $\left(A(\e D_x)+A_0/i\right)$ is
$$
A(\xi)+A_0/i
$$
which is a symmetric matrix for any $\xi\in\R^d$. Direct computation gives the following smooth spectral decomposition
\be\label{sp-A}
A(\xi)+A_0/i=\l_1(\xi) \Pi_1(\xi) +\l_2(\xi)\Pi_2(\xi)+\l_3(\xi) \Pi_3(\xi)
\ee
with the eigenvalues
\be\label{sp-A-eiv}
\l_1(\xi) =\sqrt{1+|\xi|^2}=\lxi,\quad \l_2(\xi) =-\sqrt{1+|\xi|^2}=-\lxi,\quad \l_3(\xi)\equiv 0
\ee
and eigenprojections
\be\label{sp-A-eip}
\Pi_j(\xi) =\frac{1}{2}\bp \frac{\xi\xi^T}{\l_j^2}&\frac{\xi}{\l_j}&\frac{-i\xi}{\l_j^2}\\\frac{\xi^T }{\l_j} & 1 & \frac{-i}{\l_j} \\ \frac{i\xi^T}{\l_j^2} & \frac{i}{\l_j} &\frac{1}{\l_j^2}\ep,\quad \Pi_3(\xi)=\frac{1}{d+|\xi|^2}\bp \Id_d &0 &-i\xi \\ 0&0&0\\ i\xi^T &0 & |\xi|^2 \ep,
\ee
where $j\in \{1,2\}$, $\xi=(\xi_1,\cdots,\xi_d)^T$ is a column vector and $\Id_d$ denotes the unit matrix of order $d$.  It is direct to check that $\l_{j}\in S^1$ and $\Pi_{j}\in S^0$ for any $j\in\{1,2,3\}.$  Clearly, we have $\l_j\in S^1,\ \Pi_j\in S^0,\ j\in\{1,2,3\}$. As a result, Assumption \ref{ass-spec} is satisfied.

According to \eqref{sp-A}, we can write
\be\label{sp-A-op-kg}
A(\e D_x)+A_0/i=\l_1(\e D_x) \Pi_1(\e D_x) +\l_2(\e D_x)\Pi_2(\e D_x).\nn
\ee

\subsection{WKB approximate solution}\label{sec:wkb-sol}

In this section, we use WKB expansion to construct an approximate solution to Cauchy problem \eqref{000}-\eqref{ini-data000}. Moreover, we will show that this approximate solution is global-in-time well defined and uniformly bounded.  As a result, Assumption \ref{ass-app} is verified.

\subsubsection{WKB cascade}\label{sec:wkb-1}

We make a formal power series expansion in $\e$ for the solution and each term in the series is a trigonometric polynomial in $\th:=-t/\e^2$:
\be\label{def-wkb}U_a = \sum_{n=0}^{K_a+1} \e^{n} { U}_{n}, \qquad { U}_{n} = \sum_{p \in \Z} e^{i p \th} U_{n,p}, \quad K_a \in \Z_+.\ee
 The amplitudes $U_{n,p}(t,x)$  are not highly-oscillating (independent of $\th$) and satisfy $U_{n,-p}=\overline U_{n,p}$  due to the reality of $U_a$.  We plug \eqref{def-wkb} into \eqref{000} and deduce the system of order $O(\e^n),\ n=-2,-1,0,1.$

We start from considering the equations in the terms of order $O(\e^{-2})$. We reproduce such equations as follows
\be\label{wkb--2}
(-ip+A_0)U_{0,p}=0,\quad \mbox{for all $p$}.
\ee
It is easy to find that $(-ip+A_0)$ are invertible except when $p\in {\cal H}_0:=\{-1,0,1\}$. We then deduce from \eqref{wkb--2} that
\be\label{U0p-p>2}
U_{0,p}=0,\quad \mbox {for all $p$ such that $|p|\geq 2$}.
\ee

We do not need to include the mean mode $U_{0,0}$ in the approximation. For simplicity, we take
\be\label{U00}
U_{0,0}=0.
\ee

For $p=1$, \eqref{wkb--2} is equivalent to the so called polarization condition $U_{0,p}\in \ker(ip+A)$. This implies
\be\label{intr-pola}
 U_{0,1}= g_{0} e_{1},\quad e_1:= (0_d^T,-i,1)^T, \quad g_0 \mbox{~is a scalar function}.
\ee

For $p=-1$, reality implies \be\label{U0-1}U_{0,-1}=\overline U_{0,1}=\bar g_0 e_{-1},\quad e_{-1}:=\bar e_1=(0_d^T,i,1)^T.\ee

\medskip

We continue to consider the equations in the terms of order $O(\e^{-1})$:
\be\label{wkb--1}
A(\d_x) U_{0,p}+(-ip+A_0) U_{1,p}=0, \quad \mbox{for all $p$}.
\ee

When $p=0$, by the choice of the leading mean mode in \eqref{U00}, equation \eqref{wkb--1} becomes
$$
A_0 U_{1,0}=0
$$
which is equivalent to
\be\label{U10}
U_{1,0}=(h_{1}^T,0,0)^T  \quad \mbox{for some vector valued function $h_1\in \R^d$}.
\ee

\smallskip

When $p=1$, by \eqref{intr-pola}, equation  \eqref{wkb--1} is equivalent to
\be\label{U11}
U_{1,1}=g_{1}e_1 +(\nabla^T g_{0},0,0)^T \quad \mbox{for some scalar function $g_1$}.
\ee

\smallskip

When $|p|\geq 2$, the invertibility of $(-ip+A_0)$ and \eqref{U0p-p>2} imply
\be\label{U1p-p>2}
U_{1,p}=0,\quad \mbox{for all $p$ such that $|p|\geq 2$}.
\ee

\medskip

The equations in the terms of order $O(\e^0)$ are as follows:
\be\label{wkb-0}
\d_t U_{0,p}+A(\d_x) U_{1,p}+(-ip+A_0) U_{2,p}=\sum_{p_1+p_2=p} B(U_{0,p_1},U_{0,p_2}), \quad \mbox{for all $p$}.
\ee

When $p=0$, by \eqref{def-B},  \eqref{U0p-p>2}--\eqref{U0-1}, equation \eqref{wkb-0} becomes
$$
A(\d_x) U_{1,0}+A_0 U_{2,0}=2 B(U_{0,1},U_{0,-1})=-2\l (0_d^T,|g_0|^2,0)^T
$$
which is equivalent to (by employing \eqref{def:AB} and \eqref{U10})
\be\label{U20}
U_{2,0}=(h_{2}^T,0, {\rm div}h_1-2\l |g_0|^2 )^T  \quad \mbox{for some vector valued function $h_2\in \R^d$}.
\ee

\smallskip

When $p=1$, by \eqref{def:AB}, \eqref{def-B}, \eqref{U0p-p>2} and \eqref{U00}, equation  \eqref{wkb-0} becomes
\be\label{wkb-0-U01}
\d_t U_{0,1}+A(\d_x) U_{1,1}+(-i+A_0) U_{2,1}=0.
\ee
By \eqref{intr-pola} and \eqref{U11}, equation \eqref{wkb-0-U01} is equivalent to
\be\label{wkb-0-U01-1}\left\{\begin{aligned}
&2i\d_t g_0+\Delta g_0=0,\\
&U_{2,1}=g_2 e_1 +(\nabla^T g_1,\d_t g_0,0)^T,\quad \mbox{for some scalar function $g_2$}.
\end{aligned}\right.
\ee
This is how we obtain the  linear Schr\"odinger equation \eqref{g0-sch1}. The initial datum of $g_0$
is determined in such a way that $U_{0}(0)=(0_d^T,\psi_0,\varphi_0)^T$ which is the leading term of initial data $U(0)$ (see \eqref{ini-data000}). This imposes
\be\label{wkb-ini-g0}
g_0(0)=\frac{\varphi_0+i \psi_0}{2}.
\ee

\smallskip

When $p=2$, by \eqref{def:AB}, \eqref{def-B}, \eqref{U0p-p>2}--\eqref{intr-pola},  \eqref{U1p-p>2}, equation \eqref{wkb-0} becomes
$$
(-2i+A_0) U_{2,2}= B(U_{0,1},U_{0,1})=-2\l (0_d^T,g_0^2,0)^T
$$
which is equivalent to
\be\label{U22}
U_{2,2}=\frac{\l}{3}\left(0_d^T,-2i g_0^2,g_0^2\right)^T.
\ee

\smallskip

When $|p|\geq 3$, equation \eqref{wkb-0} implies
\be\label{U2p-p>3}
U_{2,p}=0,\quad \mbox{for all $p$ such that $|p|\geq 3$}.\nn
\ee

\medskip

We finally consider the equations of order $O(\e)$:
\be\label{wkb-1}
\d_t U_{1,p}+A(\d_x) U_{2,p}+(-ip+A_0) U_{3,p}=2\sum_{p_1+p_2=p} B(U_{0,p_1},U_{1,p_2}), \quad \mbox{for all $p$}.
\ee

When $p=0$, by \eqref{def:AB},  \eqref{def-B},  \eqref{U0p-p>2}--\eqref{U0-1}, \eqref{U11}, \eqref{U1p-p>2},  equation \eqref{wkb-1} becomes
$$
\d_t U_{1,0}+A(\d_x) U_{2,0}+A_0 U_{3,0}=4\Re B(U_{0,1},U_{1,-1})=-4\l (0_d^T,\Re (g_0 \bar g_1),0)^T
$$
which is equivalent to (by \eqref{U10} and \eqref{U20})
\be\label{U30}
\d_t h_1=0,~ U_{3,0}=(h_{3}^T,0, {\rm div}h_2-4\l \Re(g_0\bar g_1))^T,
\ee
for some vector valued function $h_3\in \R^d$. The notation $\Re a$ stands for the real part of $a$.

Here we take a trivial solution $h_1=0$ to the equation $\d_t h_1 =0$ in \eqref{U30}. By \eqref{U10}, this means
\be\label{U10-0}
U_{1,0}=0.
\ee

\smallskip

When $p=1$, by  \eqref{def-B}, \eqref{U0p-p>2},  \eqref{U00}, \eqref{U1p-p>2} and \eqref{U10-0}, equation  \eqref{wkb-1} becomes
\be\label{wkb-0-U11}
\d_t U_{1,1}+A(\d_x) U_{2,1}+(-i+A_0) U_{3,1}=0\nn
\ee
which  is equivalent to
\be\label{wkb-0-U11-1}\left\{\begin{aligned}
&2i\d_t g_1+\Delta g_1=0,\\
&U_{3,1}=g_3 e_1 +(\nabla^T g_2,\d_t g_1,0)^T,\quad \mbox{for some scalar function $g_3$}.
\end{aligned}\right.\nn
\ee
Here we used \eqref{U11} and \eqref{wkb-0-U01-1}.

We find that $g_1$ satisfies the same linear Schr\"odinger equation as $g_0$. Since we do not need to include initial data of $g_1$ (this may be needed sometimes in order to have a better initial approximation), we will take a trivial solution $g_1=0$.

\smallskip

When $p=2$, by  \eqref{def:AB},  \eqref{def-B},  \eqref{intr-pola}, \eqref{U0p-p>2},   \eqref{U11}, \eqref{U1p-p>2} and \eqref{U10-0}, equation \eqref{wkb-1} becomes
$$
A(\d_x)U_{2,2}+(-2i+A_0) U_{3,2}= 2B(U_{0,1},U_{1,1})=-2\l (0_d^T,g_0g_1,0)^T
$$
which is equivalent to (by \eqref{U22})
\be\label{U32}
U_{3,2}=\frac{2\l}{3}\left(g_0 \nabla^T g_0,-2i g_0g_1,g_0g_1\right)^T.\nn
\ee

When $|p|\geq 3$, \eqref{wkb-1} is equivalent to
\be\label{U3p-p>3}
U_{3,p}=0,\quad \mbox{for all $p$ such that $|p|\geq 3$}.\nn
\ee

\subsubsection{WKB approximate solution}\label{sec:proof-prop1}
By \eqref{reg-ini}, we have $g_0(0)\in H^s$ with $s>d/2+4$. Then classically there exists a unique global-in-time solution $g_0$ to the Cauchy problem $\eqref{wkb-0-U01-1}_1$-\eqref{wkb-ini-g0} in Sobolev space $H^s$. Moreover, we have the estimate
\be\label{est-g0}
\|\d_t g_0\|_{L^\infty(0,\infty; H^{s-2})}\leq C \| g_0\|_{L^\infty(0,\infty; H^{s})}\leq C\|(\phi_0,\psi_0)\|_{H^s}.
\ee
To construct an approximate solution, we need to determine $g_j$ and $h_j$, $j\in \{1,2,3\}$, appeared in Section \ref{sec:wkb-1}. Taking
$$g_1=g_2=g_3=h_1=h_2=h_3=0$$
implies, by employing the argument in Section \ref{sec:wkb-1}, that
\ba\label{Ujp}
&U_{0,1}=g_0 e_1,\quad U_{1,1}=\bp \nabla g_{0}\\0\\0\ep,\quad U_{2,0}= -2\l \bp0_d\\0\\ |g_0|^2\ep, \\
& U_{2,1}= \bp0_d\\ \d_t g_0 \\0\ep,\quad U_{2,2}= \frac{\l}{3}\bp 0_d\\ -2i g_0^2\\ g_0^2\ep,\quad U_{3,2}= \bp g_0 \nabla g_{0}\\ 0 \\ 0 \ep,
\ea
and $U_{n,p}=0$ for all other $(n,p)\in  \Z^2,~p\geq 0$, and $U_{n,p}=\overline U_{n,-p}$ for $p<0$.

We observe that all the components in \eqref{Ujp} are determined by the leading amplitude $g_0$. By the estimate of $g_0$ in \eqref{est-g0}, we have for any $(n,p)\in \Z^2$:
\be\label{est-Unp}
U_{n,p}\in L^\infty(0,\infty;H^{s-2}),\quad \d_t U_{n,p} \in L^\infty(0,\infty;H^{s-4}).
\ee

Plugging all such $U_{n,p}$ into \eqref{def-wkb} gives an approximate solution $U_a$ of the form
\be\label{wkb-sol-010}
U_a=U_0+\e U_1 +\e^2 U_2+\e^3 U_3
\ee
 which solves the following Cauchy problem globally in time
 \be \label{wkb-eq1} \left\{ \begin{aligned} &\d_t U_a + \frac{1}{\e} A(\d_x) U_a + \frac{1}{\e^2} A_0 U_a = B(U_a,U_a)-\e^2 R^\e,\\
 & U_a(0)=(\e \nabla^T \varphi_0, \psi_0, \varphi_0)^T+\e^2 U_2(0)+\e^3 U_3(0),\end{aligned}\right.
 \ee
 where
 \ba\label{def-Re}
 R^\e:=2B(U_0,U_2)+ B(U_1,U_1)+2 \e B(U_1, U_2)+\e^2 B(U_2,U_2)\\
 - \sum_{n=2}^3 \e^{n-2}\sum_{p} e^{-ipt/\e^2} \d_t U_{n,p}-\sum_{p} e^{-ipt/\e^2} A(\d_x)U_{3,p}.
 \ea

It is direct to check:
\be\label{bound-R-Psi-kg}
\sup_{0<\e<1} \left(\|R^{\e}\|_{L^\infty\left(0,\infty; H^{s-4}\right)}+\|U(0)-U_a(0)\|_{ H^{s-4}}\right)<+\infty.
\ee

Recall $s>d/2+4$. Hence, by \eqref{est-Unp}-\eqref{bound-R-Psi-kg}, this approximate solution $U_a$ fulfills Assumption \ref{ass-app} for the Cauchy problem \eqref{000}-\eqref{ini-data000}.

\subsection{Partially strong transparency}\label{sec:very-KG}

By \eqref{def-B} and \eqref{sp-A-eip}, direct computation implies
\be\label{Pi3B}
\Pi_3(\xi)B(\cdot, \cdot)\equiv 0.\nn
\ee

By \eqref{sp-A-eiv}, we have
\be\label{res-kg1}
R_{j,j',p}=\{\xi\,:\, \l_j(\xi)-\l_{j'}(\xi)-p=0 \}\nn
\ee
are all empty sets except
\be\label{res-kg2}
R_{1,3,1}=\{\xi\,:\, \l_1(\xi)-1=0 \}=\{0\}, \quad R_{2,3,-1}=\{\xi\,:\, \l_2(\xi)+1=0\}=\{0\}.
\ee

Now we compute the interaction phases and the interaction coefficients corresponding to the non-empty resonance sets in \eqref{res-kg2}. On one hand, direct computation gives
\ba\label{inter-coe}
&\Pi_1(\xi) B(e_1)\Pi_{3}(\xi)=\frac{-\l}{2(d+|\xi|^2)}\bp \frac{i\xi\xi^T}{\l_1} &0 &\frac{\xi|\xi|^2}{\l_1} \\ \frac{i\xi^T}{\l_1} &0 &\frac{|\xi|^2}{\l_1}\\ \frac{-\xi^T}{\l_1} &0 &\frac{i|\xi|^2}{\l_1}\ep,\\
&\Pi_2(\xi) B(e_{-1})\Pi_{3}(\xi)=\frac{-\l}{2(d+|\xi|^2)}\bp \frac{i\xi\xi^T}{\l_2} &0 &\frac{\xi|\xi|^2}{\l_2} \\ \frac{i\xi^T}{\l_2} &0 &\frac{|\xi|^2}{\l_2}\\ \frac{-\xi^T}{\l_2} &0 &\frac{i|\xi|^2}{\l_2}\ep.
\ea
On the other hand, the interaction phases satisfy
\be\label{res-cal}
|\l_1(\xi)-1|^{-1}=|\l_2(\xi)-1|^{-1}=\frac{1}{\sqrt{1+|\xi|^2}-1}=\frac{\sqrt{1+|\xi|^2}+1}{|\xi|^2}.
\ee

We find that $|\Pi_1(\xi) B(e_1)\Pi_{3}(\xi)|\cdot |\l_1(\xi)-1|^{-1}$ and $|\Pi_2(\xi) B(e_{-1})\Pi_{3}(\xi)|\cdot |\l_2(\xi)+1|^{-1}$ are unbounded near resonance $\xi=0$. This implies that the strong transparency condition is not satisfied when $(i,j,p)=(1,3,1)$ or $(i,j,p)=(2,3,-1)$.

However, by \eqref{inter-coe} and \eqref{res-cal}, we can show that the following partially strong transparency condition is satisfied
\be\label{par-st-tr}
|\Pi_1(\xi) B(e_1)\Pi_3(\xi)|\leq C |\l_1(\xi)-1|^{1/2},\quad |\Pi_2(\xi) B(e_{-1})\Pi_3(\xi)|\leq C |\l_2(\xi)+1|^{1/2}.\nn
\ee

Thus,  Assumption \ref{ass-ps-tran} is satisfied with $\a=1/2$. Moreover, the eigenvalue $\l_3$ is identically zero, which shows that  Assumption \ref{ass-ps-tran-add} is also satisfied.

\subsection{Proof of Theorem \ref{thm:main}}

All the assumptions introduced in Section \ref{sec:ass-res} are verified for \eqref{000}-\eqref{ini-data000}. By applying Theorem \ref{thm:general}, we obtain
\begin{theo}\label{thm1} There exists $\e_0>0$ such that for any $0<\e<\e_0$, the Cauchy problem \eqref{000}-\eqref{ini-data000} admits a unique solution $U\in L^\infty\left(0,\frac{T}{\e}; H^{s-4}\right)$
for some $T>0$ independent of $\e$. Moreover, there holds
\be\label{error-est1}
\|U-U_a\|_{L^\infty\left(0,\frac{T}{\e}; H^{s-4}\right)}\leq C\, \e,\nn
\ee
where $U_a$ is the approximate solution obtained in Section \ref{sec:proof-prop1}.
\end{theo}

 Theorem \ref{thm:main} is a direct corollary of Theorem \ref{thm1}.


\begin{section}*{Acknowledgements}
\thanks{The authors thank Professor Weizhu Bao for helpful discussions. The first author acknowledges the support of the project
LL1202 in the programme ERC-CZ funded by the Ministry of Education, Youth and Sports of the Czech Republic.  Z. Zhang was partially supported by NSF of China under Grant 11371039 and 11425103.}
\end{section}







\end{document}